\documentclass{elsarticle}
\usepackage[utf8]{inputenc}
\usepackage{amssymb,amsmath}
\usepackage{amsthm}
\usepackage{graphicx}
\usepackage{color}
\usepackage{tikz}
\usepackage{pgfplots}
\pgfplotsset{compat=1.16}
\usepackage{enumitem}
\usepackage{setspace}
\usepackage{mathtools}
\usepackage{xfrac}	
\usepackage{siunitx}

\usepackage{multicol} 
\usepackage{soul}
\usepackage{subcaption}
\usepackage{booktabs}
\usepackage{array}
\newcolumntype{L}[1]{>{\raggedright\let\newline\\\arraybackslash\hspace{0pt}}m{#1}}
\newcolumntype{C}[1]{>{\centering\let\newline\\\arraybackslash\hspace{0pt}}m{#1}}
\newcolumntype{R}[1]{>{\raggedleft\let\newline\\\arraybackslash\hspace{0pt}}m{#1}}
\usepackage[linesnumbered, ruled,vlined]{algorithm2e}
\usepackage{changepage}
\usetikzlibrary{spy}

\usetikzlibrary{intersections}

\usepackage[hidelinks]{hyperref}

\mathtoolsset{showonlyrefs}

\makeatletter
\newcommand{\smallsym}[2]{#1{\mathpalette\make@small@sym{#2}}}
\newcommand{\make@small@sym}[2]{%
	\vcenter{\hbox{$\m@th\downgrade@style#1#2$}}%
}
\newcommand{\downgrade@style}[1]{%
	\ifx#1\displaystyle\scriptstyle\else
	\ifx#1\textstyle\scriptstyle\else
	\scriptscriptstyle
	\fi\fi
}
\makeatother

\newcommand{\supp}{\mathrm{supp}}

\newcommand{\opnorm}[1]{{\left\vert\kern-0.25ex\left\vert\kern-0.25ex\left\vert #1 
	\right\vert\kern-0.25ex\right\vert\kern-0.25ex\right\vert}}

\definecolor{myBlue}{RGB}{30,144,255}
\definecolor{myGreen}{RGB}{69,169,0}
\definecolor{myRed}{RGB}{165,12,42} 
\definecolor{myOrange}{RGB}{225,92,22} 
\definecolor{color2}{RGB}{255, 126, 126}
\definecolor{color3}{RGB}{0, 100, 0}
\definecolor{color1}{RGB}{176, 226, 255}

\tikzset{cross/.style={cross out, draw=black, minimum size=2*(#1-\pgflinewidth), inner sep=0pt, outer sep=0pt},
cross/.default={0.6ex}}

\newtheorem{theorem}{Theorem}[section]
\newtheorem{proposition}[theorem]{Proposition}
\newtheorem{lemma}[theorem]{Lemma}
\newtheorem{corollary}[theorem]{Corollary}

\theoremstyle{definition}
\newtheorem{definition}[theorem]{Definition}
\newtheorem{assumption}[theorem]{Assumption}
\newtheorem{remark}[theorem]{Remark}
\newtheorem{example}[theorem]{Example}

\pgfplotstableset{
every head row/.style={before row=\toprule, after row=\midrule},
every last row/.style={after row=\bottomrule},
}

\numberwithin{theorem}{section}
\numberwithin{equation}{section}
\numberwithin{table}{section}
\numberwithin{figure}{section}
\allowdisplaybreaks
\newcommand{\hV}{\hat{V}}
\newcommand{\hVH}{\hat{V}_H}
\newcommand{\Vms}{V^{\mathrm{ms}}_H}
\newcommand{\VmsHk}{V^{\mathrm{ms}}_{H,k}}
\newcommand{\VH}{V_H}

\renewcommand{\TH}{\mathcal{T}_H}

\newcommand{\bV}{V}

\newcommand{\bv}{v}
\newcommand{\vd}[1]{{v}_{(#1)}}

\newcommand{\bL}{L}

\newcommand{\bK}{K}

\newcommand{\ddu}{D^2_t{u}}
\newcommand{\dddu}{D^3_t{u}}
\newcommand{\ddddu}{D^4_t{u}}

\newcommand{\zHk}{z_{H,k}}
\newcommand{\zH}{z_{H}}
\newcommand{\zms}{z^{\mathrm{ms}}}
\newcommand{\zmsHk}{z^{\mathrm{ms}}_{H,k}}
\newcommand{\vHk}{v_{H,k}}
\newcommand{\vH}{v_{H}}

\newcommand{\rhok}{\rho_k}

\newcommand{\hf}{\hat{f}}
\newcommand{\ms}{\mathrm{ms}}

\newcommand{\tn}{t_n}
\newcommand{\tnhp}{t_{n+1/2}}
\newcommand{\tnhm}{t_{n-1/2}}
\newcommand{\tnnext}{t_{n+1}}
\newcommand{\tnprev}{t_{n-1}}
\newcommand{\ti}{t_i}
\newcommand{\tihp}{t_{i+1/2}}
\newcommand{\tihm}{t_{i-1/2}}
\newcommand{\tinext}{t_{i+1}}
\newcommand{\tiprev}{t_{i-1}}
\newcommand{\tohp}{t_{1/2}}
\newcommand{\unH}{u^n_H}
\newcommand{\unHhp}{u^{n+1/2}_H}
\newcommand{\unHhm}{u^{n-1/2}_H}

\newcommand{\unhp}{u^{n+1/2}}
\newcommand{\unhm}{u^{n-1/2}}
\newcommand{\un}{u^n}
\newcommand{\unprev}{u^{n-1}}
\newcommand{\unnext}{u^{n+1}}

\newcommand{\uo}{u^0}
\newcommand{\uhp}{u^{1/2}}
\newcommand{\ui}{u^i}
\newcommand{\uiprev}{u^{i-1}}

\newcommand{\bdiscd}{\bar{\partial}_t}
\newcommand{\fdiscd}{\partial_t}
\newcommand{\omni}{\omega^n_1}
\newcommand{\omnii}{\omega^n_2}
\newcommand{\omii}{\omega^i_1}
\newcommand{\omiii}{\omega^i_2}
\newcommand{\thn}{\theta^n}

\newcommand{\thnk}{\theta_k^n}

\newcommand{\thnkhm}{\theta_k^{n-1/2}}
\newcommand{\thnkhp}{\theta_k^{n+1/2}}
\newcommand{\thok}{\theta_k^0}
\newcommand{\thkhp}{\theta_k^{1/2}}

\newcommand{\rhonk}{\rho^n_k}
\newcommand{\rhonkhp}{\rho^{n+1/2}_k}

\newcommand{\rhoik}{\rho^i_k}
\newcommand{\fn}{f^n}
\newcommand{\ddf}{D^2_t{f}}

\newcommand{\Qnddu}{\Theta^{n}_{D^2_tu}}

\newcommand{\Qiddu}{\Theta^{i}_{D^2_tu}}

\newcommand{\Qndddduddf}{\Theta^n_{D^4_tu-D^2_tf}}
\newcommand{\Qidddduddf}{\Theta^i_{D^4_tu-D^2_tf}}

\newcommand{\unmsH}{u^{\mathrm{ms},n}_{H}}
\newcommand{\unmsHhp}{u^{\mathrm{ms},n+1/2}_{H}}
\newcommand{\unmsHhm}{u^{\mathrm{ms},n-1/2}_{H}}
\newcommand{\unmsHk}{u^{\mathrm{ms},n}_{H,k}}
\newcommand{\unmsHkhp}{u^{\mathrm{ms},n+1/2}_{H,k}}
\newcommand{\unmsHkhm}{u^{\mathrm{ms},n-1/2}_{H,k}}

\newcommand{\enhp}{e^{n+1/2}}

\newcommand{\ddt}{\frac{\mathrm{d}}{\mathrm{d}t}}
\newcommand{\Rms}{R^\mathrm{ms}}
\newcommand{\Rmsk}{R^\mathrm{ms}_k}
\newcommand{\Qk}{\mathcal{Q}_k}

\newcommand{\PH}{P_H}
\newcommand{\cI}{\mathcal{I}}
\newcommand{\sumin}{\sum_{i=1}^n}

\newcommand{\intt}{\int_0^t}

\newcommand{\dt}{\mathrm{d}t}
\newcommand{\ds}{\mathrm{d}s}

\newcommand{\knorm}[1]{\big| #1 \big|_K}
\newcommand{\lnorm}[1]{\big| #1 \big|_L}
\newcommand{\mnorm}[1]{\big| #1 \big|_M}

\newcommand{\intmnorm}[1]{\bigg| #1 \bigg|_M}

\newcommand{\lonemnorm}[1]{\big| #1 \big|_{L^1(M)}}
\newcommand{\linfmnorm}[1]{\big| #1 \big|_{L^\infty(M)}}
\newcommand{\loneknorm}[1]{\big| #1 \big|_{L^1(K)}}
\newcommand{\linfknorm}[1]{\big| #1 \big|_{L^\infty(K)}}

\newcommand{\calG}{\mathcal{G}}
\newcommand{\calN}{\mathcal{N}}
\newcommand{\calE}{\mathcal{E}}
\newcommand{\calQ}{\mathcal{Q}}
\newcommand{\Rd}{\mathbb{R}^d}

\newcommand{\com}[1]{{\color{black} #1}}

\begin{document}
\title{Multiscale methods for solving wave equations on spatial networks}

\author[1,2]{Morgan G\"ortz\fnref{fn1}}
\ead{morgan.gortz@fcc.chalmers.se}

\author[2]{Per Ljung\corref{cor1}}
\ead{perlj@chalmers.se}

\author[2]{Axel M{\aa}lqvist\fnref{fn2}}
\ead{axel@chalmers.se}

\cortext[cor1]{Corresponding author}

\fntext[fn1]{Supported by the Swedish Foundation for Strategic Research (SSF)}
\fntext[fn2]{Supported by the G\"oran Gustafsson Foundation for Research in Natural Sciences and Medicine and the Swedish Research Council (project number 2019-03517)}
\fntext[fn3]{The computations were performed on resources provided by the Swedish National Infrastructure for Computing (SNIC) at C3SE.}

\affiliation[1]{organization={Computational Engineering
		and Design, Fraunhofer-Chalmers Centre},
	addressline={Sven Hultins gata 9},
	postcode={412 58},
	city={Gothenburg},
	country={Sweden}}

\affiliation[2]{organization={Department of Mathematical Sciences, Chalmers University of Technology and University of Gothenburg},
	addressline={Chalmers tvärgata 3},
	postcode={412 58},
	city={Gothenburg},
	country={Sweden}}

\date{\today}
%
%
\begin{abstract}

We present and analyze a multiscale method for wave propagation problems, posed on spatial networks. By introducing a coarse scale, using a finite element space interpolated onto the network, we construct a discrete multiscale space using the localized orthogonal decomposition (LOD) methodology. The spatial discretization is then combined with an energy conserving temporal scheme to form the proposed method. Under the assumption of well-prepared initial data, we derive an a priori error bound of optimal order with respect to the space and time discretization. In the analysis, we combine the theory derived for stationary elliptic problems on spatial networks with classical finite element results for hyperbolic problems. Finally, we present numerical experiments that confirm our theoretical findings.

\end{abstract}
%
%
\maketitle
%

\section{Introduction}\label{s:intro}

The modeling of physical phenomena using partial differential equations (PDEs) is a major topic in science and engineering. Numerical simulation of these models is often challenging due to their great complexity. In some applications it is possible to introduce a so called spatial network model in order to reduce complexity, while still maintaining the main features of the original model. Flow in porous rock formations \cite{ChuEP12} and modeling of fiber based elastic materials \cite{IliLW10} are two examples where this technique is commonly applied. More specifically, if we consider a fiber based material consisting of three dimensional cylinders, we may replace these by one dimensional beams that together form a web like spatial network, defined by nodes, edges and edge weights. However, even if the complexity of the model is significantly reduced, there are still challenges present when solving equations on the resulting spatial network, in particular if the corresponding network data (edge weights) are rapidly varying in space. For the fiber based model, the network inherits such variations from the geometry and distribution of fibers. In the case of flow through a porous medium, crack formations and variation in rock permeability  contributes to such large spatial variations in the network data.

In this paper, we study wave propagation on spatial networks. We consider a symmetric system matrix $K$, related to a weighted graph Laplacian with high data variation, and a diagonal mass matrix $M$. The displacement $u$ fulfills
\begin{align}
	MD_t^2 u+Ku=Mf,
\end{align}
given some right-hand side $f$ and appropriate initial and boundary data. This equation can, e.g.,~arise from a finite element discretization of the acoustic wave equation posed on a web of connected one dimensional line segments, see Figure \ref{fig:mean and std of nets}. If we allow vector valued solutions it may model elastic wave propagation on a fiber based material. Prior to our work, wave equations posed on graphs have been studied thoroughly in the literature. For instance, existence and uniqueness of solutions to linear and non-linear wave equations on graph networks are studied in \cite{LinX19, LinX22} and solution properties such as eigenvalues, stability, periodicity and propagation speed, are discussed in, e.g., \cite{Sch09, Kha18}. In \cite{Tay99}, the importance of elastic wave modeling in porous medium for seismic exploration is emphasized, and the author defines a pore network model for the analysis of wave properties. 

From a numerical point of view, complications arise when the micro-scale features in the network model affect the solution on the macro scale, and must be resolved in order to obtain accurate approximations. In the research on numerical methods for solving partial differential equations, there have been extensive investigations into how to circumvent this problem, mainly by so-called multiscale methods. In short, multiscale methods aim to construct coarse-scale spaces whose approximation properties are improved in comparison to classical finite element spaces. This is generally achieved by incorporating problem-dependent information on a fine scale into the coarse-scale space. Notable examples of such methods include Generalized multiscale finite element method (GMsFEM) \cite{BabL11, BabO83, EfeGH13}, gamblets \cite{Owh17}, the Localized Orthogonal Decomposition (LOD) method \cite{HenP13, MalP14}, and the Super-LOD method \cite{HauP21, FreHP21}. For multiscale methods specifically designed for wave equations, we refer to \cite{AbdG11, EngHR11, ArjR14, AbdGJ18, GaoFC18}, and for more information on techniques based on numerical homogenization in general, see, e.g., \cite{MalP20, AltHP21, OwhS19}.

The purpose of this paper is to construct and analyze a LOD based method that efficiently and accurately approximates solutions to a linear wave equation, posed on a spatial network. LOD is by now a well-established methodology in the area of multiscale numerical analysis. It was first introduced in \cite{MalP14} and has since then been further developed for various types of equations, such as elliptic-type in \cite{HenM14, HenMP14} and parabolic-type in \cite{MalP17, MalP18, LjuMM22, AltCMPP20}. In particular, the LOD technique for wave equations is analyzed in \cite{LjuMP20, AbdH17, PetS17, MaiP19, GeeM21}, and its adaptation to spatial network models was first considered in \cite{KetMMFWE20}, with theoretical justifications in \cite{EdeGHKM22}. The LOD method is based on a decomposition of the solution space into a coarse-scale and a fine-scale part respectively. In a PDE setting, the coarse and fine scales appears naturally using nested finite element spaces. Finding coarse scales in spatial networks is less obvious. In \cite{EdeGHKM22} and \cite{GorHM22}, the authors address this issue, and find a solution by introducing an artificial coarse scale with minimum assumptions on its relation to the network. By incorporating fine scale features into the coarse-scale space, we obtain a modified space with enhanced approximation properties. We remark that the construction of the modified space is feasible in terms of computational complexity, but challenging due to the computations made on the fine scale. However, for time-dependent problems, the LOD technique is significantly enhanced, in the sense that the fine-scale features need only be solved for once, and can then be re-used in each time step.

In combination with the spatial discretization, an energy conserving temporal scheme is deployed to construct the fully discretized method. For the proposed method, we state and prove a priori optimal order convergence with respect to both space and time, under the assumption of so called well-prepared initial data. To derive a priori error bounds, we utilize the theory established for stationary elliptic spatial network models, \cite{EdeGHKM22, GorHM22}, and combine these results with numerical PDE techniques for hyperbolic problems. We further present numerical experiments that confirm our theoretical findings. In a final numerical example we consider elastic wave propagation in a spatial network modeling a fiber based material.

The paper is organized as follows. In Section \ref{s:notation} we discuss the network setting, including notation and assumptions. Section \ref{s:model_problem} introduces the linear wave equation on the spatial network as our model problem, and proves regularity results for the solution. The derivation of the novel method is presented in Section \ref{s:localized_orthogonal_decomposition}, and corresponding error analysis is done in Section \ref{s:error_estimates}. Finally, we provide numerical examples that confirm our theoretical findings in Section \ref{s:numerical_examples}. 

\section{Network notation and assumptions}\label{s:notation}

Throughout the paper, we consider a network setting similar to the one presented in \cite{EdeGHKM22, GorHM22}. That is, we consider a spatial network represented by the connected graph $\calG = (\calN, \calE)$, where $\calN \subset \Rd$ is a finite set of nodes, and
\begin{align}
	\calE = \big\{\{x, y\} : \text{ $x$ and $y$ are connected by an edge}\big\}
\end{align}
is an edge set of unordered edge pairs. We introduce the notation $x\sim y$ to imply that $\{x, y\} \in \calE$. In such cases, we call $x$ and $y$ adjacent nodes, with $|x-y|$ denoting the length (Euclidean norm) of their connecting edge. The considered network is embedded into a spatial domain $\Omega \subset \Rd$, where we denote by $\Gamma \subseteq \partial \Omega$ the set of nodes on the boundary to which we apply a Dirichlet type condition. The set $\Gamma$ is assumed to be non-empty, i.e., it contains at least one node. Moreover, for simplicity, we assume a hypercube domain $\Omega = [0,1]^d$. Let $\hV$ denote the space of real functions defined on $\calN$, and let $V$ be the function space $\hV$ with imposed homogeneous Dirichlet boundary condition, i.e.,
\begin{align}
	V := \{v\in \hV: v(x) = 0,\ x\in \Gamma\}.
\end{align}
Given a subset $\omega \subset \Omega$, we define the subset of nodes contained in $\omega$ as $\calN(\omega) := \calN \cap \omega$, where we will use the abbreviation $\calN(x) = \calN(\{x\})$ for any node $x\in \calN$. The standard inner product on $\hV$ on a subset $\omega$ is then defined as
\begin{align}
	(u,v)_\omega := \sum_{x\in \calN(\omega)} u(x) v(x),
\end{align}
with $(\cdot,\cdot) = (\cdot, \cdot)_\Omega$. 

We continue by defining some crucial operators that are frequently used throughout the paper. First of all, for $x\in \calN$, let $M_x : \hV \rightarrow \hV$ be the diagonal linear operator defined by
\begin{align}
	(M_xv,v) = \frac{1}{2}\sum_{y\sim x} |x - y|v(x)^2.
\end{align}
For subdomains $\omega \subset \Omega$, we abbreviate $M_\omega:= \sum_{x\in \calN(\omega)} M_x$, and for the full domain we set $M := M_\Omega$. The operator $M$ further defines a norm, namely $\mnorm{v} := (Mv,v)^{1/2}$. Here, one can interpret $\mnorm{1}^2$, i.e., the squared $M$-norm of the constant function $1\in \hV$, as the mass of the network. 

We introduce the reciprocal edge-length weighted graph Laplacian as the operator $L_x : \hV \rightarrow \hV$ defined by
\begin{align}
	(L_xv,v) = \frac{1}{2} \sum_{y\sim x} \frac{\big(v(x) - v(y) \big)^2}{|x - y|},
\end{align}
with similar abbreviations $L_\omega:= \sum_{x\in \calN(\omega)}L_x$ and $L := L_\Omega$. The weighted Laplacian $L$ is a symmetric and positive semi-definite operator, and hence defines the semi-norm $\lnorm{v} := (Lv,v)^{1/2}$.

At last, we introduce the symmetric and positive semi-definite linear operator $$K: V\rightarrow V,$$ which is to be used in our model problem. We assume $K$ to be bounded and coercive with respect to the operator $L$, i.e.,
\begin{align}
\alpha (\bL \bv, \bv) \leq (\bK \bv, \bv) \leq \beta (\bL \bv, \bv), \quad \forall \bv \in \bV,
\label{eq:KL_relation}
\end{align}
where $0 < \alpha \leq \beta < \infty$. We further assume that $K$ is local in the sense that it can be written as $K=\sum_{x\in \mathcal{N}} K_x$, with $K_x:\hat V\rightarrow \hat V$ being symmetric, positive semi-definite, with support on nodes adjacent to $x$. Since $\Gamma \neq \emptyset$, i.e., there is at least one fixed node, and since $\calG$ is connected, it holds that the operator $L$, and therefore also $K$, is invertible. The operator $\bK$ moreover defines the norm $\knorm{\bv} := (\bK \bv, \bv)^{1/2}$.

\begin{example} \label{example:weighted_laplacian}
	Let $K = \sum_{x\in \calN} K_x$, where
	\begin{align}
	(K_x v, v) = \frac{1}{2}\sum_{y\sim x} \gamma_{xy} \frac{(v(x) - v(y))^2}{|x - y|}.
	\end{align}
	Here, $\gamma_{xy} \in (0, \infty)$ is a material parameter for the connecting edge $\{x,y\}$, e.g., heat conductivity in the case of heat transfer. This operator $K$ satisfies all assumptions made above, with $\alpha = \min_{x,y\in\calN} \gamma_{xy} > 0$ and $\beta= \max_{x,y\in \calN} \gamma_{xy}$. We note that $K$ corresponds to a one dimensional finite element stiffness matrix with a varying diffusion parameter $\gamma$ that is piecewise constant on the edges of the network. The matrix $M$ corresponds to the lumped one dimensional finite element mass matrix.
\end{example}

\begin{remark} \label{remark:vec-not}
	The model and corresponding analysis presented in this paper can further be extended to vector valued functions, by letting 
	\begin{align}
	\hat{\mathbf{V}} := \hat{V}^n  = \underbrace{\hat{V}\times\cdots \times \hat{V}}_{n \text{ times}}
	\end{align}
	with elements written as $\mathbf{v}= [\vd{1}, \ldots, \vd{n}]$, where $\vd{j}\in \hV$ for $j=1,\ldots,n$, and further defining operators $\mathbf{M}$ and $\mathbf{L}$ by applying corresponding scalar valued operators component-wise. There is also the possibility to define a vector-valued operator $\mathbf{K}$ to model, e.g., elastic waves. However, we remark that a vector-valued operator has a larger kernel, and therefore requires more nodes in $\Gamma$ in order to be invertible. For more information on the vector valued setting, we refer to the work done in \cite{GorHM22} and \cite{EdeGHKM22}. We return to this setting in the final numerical example in Section \ref{s:numerical_examples}.
\end{remark}

For the error analysis, it will be convenient to use the Bochner spaces, with norms defined as
\begin{align}
	\big| v \big|_{L^1(0,T;\mathcal{B})} &:= \int_0^T \big| v(t) \big|_\mathcal{B}\, \dt, \\
	\big| v \big|_{L^\infty(0,T;\mathcal{B})} &:= \max_{[0,T]} \big| v(t) \big|_\mathcal{B},
\end{align}
where $\mathcal{B}$ will be replaced by $M$, $L$, or $K$. The Bochner space norms will be commonly abbreviated by omitting the temporal interval, such that we write, e.g., $L^1(M) := L^1(0,T;M)$.

At last, we state necessary properties of the network for the analysis to hold. For this purpose, we first introduce boxes $B_R(x)\subset \Omega$, centered at $x = (x_1, \ldots, x_d)$ with length $2R$. These are constructed by letting 
\begin{align}
\tilde{B}_R(x) := [x_1 - R, x_1 + R) \times \cdots \times [x_d - R, x_d + R),
\end{align}
and then defining
\begin{align}
B_R(x) := \tilde{B}_R(x) \cup \big( \overline{\tilde{B}_R(x)}\cup \partial \Omega \big).
\label{eq:BR_definition}
\end{align}
In this way, $B_R(x)$ does not include its upper bound in any dimension, except for the cases when the boundary is part of the global boundary $\partial \Omega$. We are now prepared to introduce following assumptions on the network.

\begin{assumption}[Network properties]
	\label{ass:network_properties}
	There is a homogeneity parameter, $R_0$, on the coarse length-scale, such that the network satisfies
	\begin{enumerate}[label=\arabic*.]
		\setlength{\itemsep}{9pt}
		\item (homogeneity) for all $R \geq R_0$ and any $x\in \Omega$, there is a uniformity constant $\sigma$ and a network density $\rho$, such that
		\begin{align}
			\rho \leq (2R)^{-d}\big| 1 \big|^2_{M,B_R(x)} \leq \sigma\rho,
		\end{align}
		\item (connectivity) for all $R\geq R_0$, $x\in \Omega$, there exists a connected subgraph $\bar{\calG} = (\bar{\calN}, \bar{\calE})$, containing\\[-0.3cm]
		\begin{enumerate}[label=(\alph*)]
			\setlength{\itemsep}{3pt}
			\item all edges with at least one endpoint in $B_R(x)$,
			\item no edges with at least one endpoint not in $B_{R+R_0}(x)$,
		\end{enumerate}
		\item (locality) the maximum edge length is smaller than $R_0$, i.e.,
		\begin{align}
			\max_{\{x,y\}\in \calE} |x - y| < R_0,
		\end{align}
		\item (boundary density) for all $y\in \Gamma$, there exists $x\in \calN(\Gamma)$ such that $|x-y| < R_0$.
	\end{enumerate}
\end{assumption}

By assuming homogeneity, we require the density of a small portion of network, i.e., on the $R_0$-scale, to be comparable with the density of the network on the entire domain $\Omega$. Moreover, the locality assumption prohibits the network from having single edges connecting over distances larger than $R_0$. At last, the connectivity property implies that the network, on the $R_0$-scale, is sufficiently connected. The connectivity assumption can further be used to derive certain inequalities valid on the network, such as analogous versions of Friedrich and Poincaré inequalities (see \cite[Lemma 3.2]{EdeGHKM22} for these results). They are in turn used to prove stability of interpolation onto the network, stated later on in Lemma \ref{lem:interpolation_bound}, which is crucial for the analysis made in this paper. For more discussion on the assumptions, and in particular the connectivity, see \cite[Section 3]{EdeGHKM22}.

\section{Model problem and finite elements}\label{s:model_problem}

We consider the problem (on strong form) to find $u(t)\in V$ such that
\begin{align}
\begin{split}
M\ddu  + Ku &= Mf, \\
u(\cdot, 0) &= g, \\
D_t u(\cdot, 0) &= h,
\label{eq:refmethod}
\end{split}
\end{align}
with corresponding weak form 
\begin{align}
\label{eq:reference_equation_weak_form}
(M\ddu, v) + (Ku,v) &= (Mf, v), \quad \forall v\in V.	
\end{align}
We remark that existence and uniqueness of a global solution to the model problem is guaranteed by the Picard--Lindelöf theorem. This result follows by writing \eqref{eq:refmethod} as a system of first order differential equations, i.e.,
\begin{align}
	D_ty(t) = \tilde{f}(t,y(t)) = 
	\begin{bmatrix}
		0 & I \\ M^{-1}K & 0
	\end{bmatrix} y(t) + 
	\begin{bmatrix}
		0 \\ f(t)
	\end{bmatrix}
\end{align}
with $y(t) = \begin{bmatrix}
u(t) & D_tu(t)
\end{bmatrix}^\mathrm{T}$, where the right-hand side $\tilde{f}(t,y(t))$ is globally Lipschitz continuous with respect to $y(t)$.

Before discretizing \eqref{eq:reference_equation_weak_form} in time and space, we begin by deriving necessary results on regularity for the solution.

\subsection{Regularity}
We define iteratively 
\begin{align}
w_0 := g, \quad w_1 := h, \quad w_j := D_t^{j-2}f(0) - M^{-1}Kw_{j-2}, \ j=2,3,\ldots,
\label{eq:initial_data}
\end{align}
and consider the following equation on strong form
\begin{align}
MD^{m+2}_t u(t) + KD^m_t u(t) &= MD^m_t f(t), \\
D^m_tu(0) &= w_m, \\
D^{m+1}_tu(0) &= w_{m+1},
\end{align}
with corresponding weak form
\begin{align}
(MD^{m+2}_tu(t), v)  + (KD^m_tu(t), v) = (MD^m_tf(t), v), \quad \forall v\in V.
\label{eq:ref_method_with_higher_derivatives}
\end{align}
We deduce following lemma on regularity of the solution.
\begin{lemma}\label{lem:regularity_result}
	Let $u(t)\in V$ be the solution to \eqref{eq:ref_method_with_higher_derivatives} for $m=0,1,2,\ldots$. Then $u$ satisfies the regularity estimate
	\begin{align}
	\mnorm{D^{m+1}_t{u}(t)} + \knorm{D^{m}_tu(t)} \leq 2\sqrt{2}\Big(\mnorm{w_{m+1}} + \knorm{w_m} + 2\lonemnorm{D^{m}_tf} \Big).
	\end{align}
\end{lemma}
\begin{proof}
	Choose test function $v = D^{m+1}_t u$ in \eqref{eq:ref_method_with_higher_derivatives}, and we have
	\begin{align}
	\frac{1}{2}\ddt \big( \mnorm{D^{m+1}_t{u}}^2 + \knorm{D^m_tu}^2 \big) \leq \mnorm{D^{m}_tf} \mnorm{D^{m+1}_t{u}}.
	\end{align}
	Integrating from $0$ to $t$, and applying Young's inequality, further gives
	\begin{align}
	\mnorm{D^{m+1}_t{u}(t)}^2 + \knorm{D^{m}_tu(t)}^2 &\leq \mnorm{D^{m+1}_t{u}(0)}^2 + \knorm{D^{m}_tu(0)}^2 + 2\intt \mnorm{D^{m}_tf} \mnorm{D^{m+1}_t{u}}\, \ds \\
	&\leq \mnorm{w_{m+1}}^2 + \knorm{w_m}^2 + \linfmnorm{D^{m+1}_t{u}} \cdot 2 \lonemnorm{D^{m}_tf} \\
	&\leq  \mnorm{w_{m+1}}^2 + \knorm{w_m}^2 + \frac{1}{2}\linfmnorm{D^{m+1}_t{u}}^2 + 2 \lonemnorm{D^{m}_tf}^2.
	\end{align}
	From here, we can take the maximum on both sides of the inequality, and move the $\linfmnorm{D_t^{m+1}u}$-term to the left-hand side, to get
	\begin{align}
		\frac{1}{2}\linfmnorm{D^{m+1}_t{u}}^2 \leq  \mnorm{w_{m+1}}^2 + \knorm{w_m}^2 +  2 \lonemnorm{D^{m}_tf}^2.
	\end{align}
	By inserting this, we find the total estimate 
	\begin{align}
	\mnorm{D^{m+1}_t{u}(t)} + \knorm{D^{m}_tu(t)} &\leq 2\sqrt{2}\Big(\mnorm{D^{m+1}_t{u}(0)} + \knorm{D^{m}_tu(0)} + 2\lonemnorm{D^{m}_tf} \Big) \\
	&= 2\sqrt{2}\Big(\mnorm{w_{m+1}} + \knorm{w_m} + 2\lonemnorm{D^{m}_tf} \Big),
	\end{align}
	where the $\sqrt{2}$ comes from removing the squares.
\end{proof}

In the error analysis presented in Section \ref{s:error_estimates}, we derive an optimal order convergence for the novel method presented in this paper. However, the estimate includes several higher orders of temporal derivatives of the solution, why the result in Lemma \ref{lem:regularity_result} is of importance. The bounds are consequently stated in terms of initial data $w_j$ and derivatives of the source function $D^j_t f$, and we require that these functions behave nicely in corresponding norms. For this purpose, we introduce the following definition on well-prepared and compatible data, which works as a network analogy of \cite[Definition 4.5]{AbdH17}.

\begin{definition}(Compatibility and well-preparedness of data)\label{def:well_prepared}
	Let $m\in \mathbb{N}$, and let $\{w_j\}_{j=0}^{m+1}$ be as defined in \eqref{eq:initial_data}. We say that the data is well-prepared and compatible of order $m$, if
	\begin{align}
		\sum_{j=0}^m \lonemnorm{D^j_t f} + \sum_{j=0}^m\knorm{w_j} + \mnorm{w_{m+1}} \leq C_{w,f},
	\end{align}
	where $C_{w,f}$ is independent of the complex features that arise from the network.
\end{definition}

\begin{remark}
	We note that if $g=h=D^jf(0)=0$ for $j=0,\ldots,m-2$, the well-preparedness and compatibility of order $m$ is trivially fulfilled. Also, by simply having initial data that makes the problem \eqref{eq:refmethod} well-posed, we immediately have well-preparedness of order 0. In other cases, we emphasize that $C_{w,f}$ is in fact computable, and it can therefore always be examined a priori that the data satisfies the required well-preparedness for the theoretical results to hold.
\end{remark}

\subsection{Temporal discretization}

We begin the discretization by introducing the temporal scheme to be used for our method (see \cite[Section 13]{Larsson03}), along with necessary energy conservation and estimation results. Let $0=:t_0 < t_1 < \ldots < t_N := T$ be a partition with uniform time step $\tau = \tn - \tnprev$, and denote $\un = u(\tn)$. The temporally discrete version of \eqref{eq:refmethod} is defined as: find $\un \in V$ such that
\begin{align}
(M\fdiscd\bdiscd \un, v) + (K\tfrac{1}{2}(\unhp + \unhm), v) = (Mf^n, v), \quad \forall v\in V,
\label{eq:discretemethodref}
\end{align}
for $n \geq 2$, with initial values $\uo, u^1 \in V$.  Here $\fdiscd \un = (\unnext - \un)/\tau$ and $\bdiscd \un = (\un - \unprev)/\tau$ denote the forward and backward discrete derivative of $\un$, respectively, and $\unhp:= (\unnext+\un)/2$. The average taken in the $(K\cdot,\cdot)$-term is done for stability and accuracy purposes. We continue by investigating how this particular average relates to the term $u(\tn)$. By Taylor expansion, 
\begin{align}
u(\tnhp) &= u(\tn) + D_t{u}(\tn)\frac{\tau}{2} + \int_{\tn}^{\tnhp} \ddu(s)(\tnhp-s)\, \ds, \\
u(\tnhm) &= u(\tn) - D_t{u}(\tn)\frac{\tau}{2} + \int_{\tnhm}^{\tn} \ddu(s)(s-\tnhm)\, \ds.
\end{align}
By adding these two lines, we see that the average is related to $u(\tn)$ as
\begin{align}
\frac{1}{2}\big( u(\tnhp) + u(\tnhm) \big) = u(\tn) + \frac{1}{2}\Qnddu.
\end{align}
Here, for convenience later on, we have defined
\begin{align}
\Theta^n_g := \int_{\tnhm}^{\tnhp} g(s)\Lambda_{[\tnhm,\tnhp]}(s)\, \ds,
\label{eq:Theta}
\end{align}
where $\Lambda_{[a,b]}$ is a hat function in time over the interval $[a,b]$, i.e.,
\begin{align}
\Lambda_{[a,b]}(t) := \begin{cases}
t-a, \quad t\in [a,\tfrac{a+b}{2}], \\
b-t, \quad t\in [\tfrac{a+b}{2}, b],
\end{cases}
\label{eq:lambda_function}
\end{align}
such that $|\Lambda_{[a,b]}| \leq (b-a)/2$.

Following the calculations in \cite[Lemma 13.2]{Larsson03}, the following energy conservation law is deduced.

\begin{lemma}
	The solution of \eqref{eq:discretemethodref}, with source function $f=0$, satisfies
	\begin{align}
	\mnorm{\fdiscd \un}^2 + \knorm{\unhp}^2 = \mnorm{\fdiscd \uo}^2 + \knorm{\uhp}^2, \quad \text{for $n\geq 0$}.
	\end{align}
\end{lemma}

We furthermore derive stability estimates for the solution $\un$ of \eqref{eq:discretemethodref} in terms of given data. The result in the following lemma will be useful later on for the error estimates in Section \ref{s:error_estimates}.

\begin{lemma}\label{lem:discreteschemeestimate}
	The solution $\un \in V$ to \eqref{eq:discretemethodref} satisfies the estimate
	\begin{align}
	\mnorm{\fdiscd \un} + \knorm{\unhp} \leq C\Big(\mnorm{\fdiscd \uo} + \knorm{\uhp} + \sumin \tau \mnorm{f^i}\Big).
	\end{align}
\end{lemma}
\begin{proof}
	We consider \eqref{eq:discretemethodref} with test function
	\begin{align}
	v = \frac{1}{2\tau}\big(\unnext - \unprev\big) = \frac{1}{2} \big(\fdiscd \un + \fdiscd \unprev \big) = \frac{1}{\tau}\big(\unhp - \unhm\big).
	\end{align}
	The first term on the left-hand side can then be written as
	\begin{align}
	(M\fdiscd \bdiscd \un, v) &= \frac{1}{2\tau}(M(\fdiscd \un - \fdiscd \unprev), \fdiscd \un + \fdiscd \unprev) = \frac{1}{2}\bdiscd \mnorm{\fdiscd \un}^2,
	\end{align}
	and the second as
	\begin{align}
	(K(\tfrac{1}{2}(\unhp + \unhm)), v) &= \frac{1}{2\tau} (K(\unhp + \unhm), \unhp - \unhm)  \\ &= \frac{1}{2}\bdiscd \knorm{\unhp}^2.
	\end{align}
	Inserting this, we find that
	\begin{align}
	\frac{1}{2}\bdiscd \big( \mnorm{\fdiscd \un}^2 + \knorm{\unhp}^2 \big) = \frac{1}{2}(Mf^n, \fdiscd \un + \fdiscd \unprev ) \leq \frac{1}{2}\mnorm{f^n} \mnorm{\fdiscd \un + \fdiscd \unprev}.
	\end{align}
	Now multiply both sides by $\tau$ and sum over $i=1,\ldots,n$, and we arrive at
	\begin{align}
	\mnorm{\fdiscd \un}^2 + \knorm{\unhp}^2 &\leq  \mnorm{\fdiscd \uo}^2 + \knorm{\uhp}^2 + \sumin \tau\mnorm{f^i} \mnorm{\fdiscd \ui + \fdiscd \uiprev}.
	\label{eq:equality}
	\end{align}
	We furthermore note that
	\begin{align}
	\sumin \tau \mnorm{f^i} \mnorm{\fdiscd \ui + \fdiscd \uiprev} &\leq \max_{i=1,\ldots,n} \mnorm{\fdiscd \ui + \fdiscd \uiprev} \sumin \tau \mnorm{f^i} \\
	&\leq 2\Big(\sumin \tau \mnorm{f^i}\Big)^2 + \frac{1}{8}\Big(\max_{i=1,\ldots,n} \mnorm{\fdiscd \ui + \fdiscd \uiprev} \Big)^2 \\
	&\leq 2\Big(\sumin \tau \mnorm{f^i}\Big)^2 + \frac{1}{2}\Big(\max_{i=0,\ldots,n} \mnorm{\fdiscd \ui} \Big)^2,
	\end{align}
	where in the second step we applied Young's weighted inequality and in the last step the triangle inequality. We can further estimate the last term in above inequality by passing it to the left-hand side in \eqref{eq:equality}, and get
	\begin{align}
	\frac{1}{2}\Big(\max_{i=0,\ldots,n} \mnorm{\fdiscd \ui}\Big)^2 \leq \mnorm{\fdiscd \uo}^2 + \knorm{\uhp}^2 + 2\Big(\sumin \tau \mnorm{f^i}\Big)^2.
	\end{align}
	In total, this yields the estimate
	\begin{align}
	\mnorm{\fdiscd \un}^2 + \knorm{\unhp}^2 \leq 2\mnorm{\fdiscd \uo}^2 + 2\knorm{\uhp}^2 + 4\Big(\sumin \tau \mnorm{f^i}\Big)^2.
	\end{align}
	Losing the squares on each term now yields the desired result.
\end{proof}

\subsection{Spatial discretization}

 We introduce a coarse finite element mesh of square shaped elements, similar to the one used for the network models in, e.g., \cite{EdeGHKM22, GorHM22}. That is, let $B_R(x)$ be the boxes defined in \eqref{eq:BR_definition}. Moreover, let $\TH$ for $H=2^{-i}$, $i\in \mathbb{N}$, be a family of partitions of $\Omega$ into uniform hypercubes (i.e., squares for $d=2$, cubes for $d=3$, etc.) of length $H$. That is,
\begin{align}
	\TH := \{ B_{H/2}(x) : x = (x_1, \ldots, x_d) \in \Omega \text{ and } H^{-1}x_i+1/2 \text{ are integers for $i=1,\ldots,d$} \}.
\end{align}
The full procedure of introducing a coarse finite element mesh of boxes is visualized in Figure~\ref{fig:mean and std of nets}.

\pgfmathsetseed{15}
\begin{figure*}
	\centering
	\begin{subfigure}[b]{0.475\textwidth}
		\centering
		\begin{tikzpicture}[scale=5.5]
\clip (-0.02, -0.02) rectangle (1.02,1.02);
    \foreach \i [count=\ni] in {1,...,3}
    {
    \node[coordinate] (l\ni) at (-0.2+0.1*rand, 0.5+0.4*rand) {};
    \node[coordinate] (r\ni) at (1.2+0.1*rand, 0.5+0.4*rand) {};
    \node[coordinate] (u\ni) at (0.5+0.4*rand, 1.2+0.1*rand) {};
    \node[coordinate] (d\ni) at (0.5+0.4*rand, -0.2+0.1*rand) {};
    
    \draw[line width=0.00mm, name path global/.expanded=lr\ni] (l\ni) -- (r\ni);
    \draw[line width=0.00mm, name path global/.expanded=lu\ni] (l\ni) -- (u\ni);
    \draw[line width=0.00mm, name path global/.expanded=ld\ni] (l\ni) -- (d\ni);
    \draw[line width=0.00mm, name path global/.expanded=ud\ni] (u\ni) -- (d\ni);
    \draw[line width=0.00mm, name path global/.expanded=ur\ni] (u\ni) -- (r\ni);
    \draw[line width=0.00mm, name path global/.expanded=dr\ni] (d\ni) -- (r\ni);
    }    
    
    \node[coordinate] (x1) at (-0.05, -0.15) {};
    \node[coordinate] (y1) at (0.95, 1.05) {};
    \draw[line width=0.00mm, name path global/.expanded=xy1] (x1) -- (y1);
    
    \node[coordinate] (x2) at (0.46, 1.05) {};
    \node[coordinate] (y2) at (0.44, -0.05) {};
    \draw[line width=0.00mm, name path global/.expanded=xy2] (x2) -- (y2);
    
    \node[coordinate] (x3) at (0.6, -0.05) {};
    \node[coordinate] (y3) at (1.05, 0.95) {};
    \draw[line width=0.00mm, name path global/.expanded=xy3] (x3) -- (y3);
    
    \node[coordinate] (x4) at (0.15, 1.05) {};
    \node[coordinate] (y4) at (0.08, -0.05) {};
    \draw[line width=0.00mm, name path global/.expanded=xy4] (x4) -- (y4);
    
    \node[coordinate] (x5) at (-0.05, 0.38) {};
    \node[coordinate] (y5) at (1.05, 0.62) {};
    \draw[line width=0.00mm, name path global/.expanded=xy5] (x5) -- (y5);
    
    \node[coordinate] (x6) at (-0.05, 0.15) {};
    \node[coordinate] (y6) at (1.05, 0.20) {};
    \draw[line width=0.00mm, name path global/.expanded=xy6] (x6) -- (y6);
    
    \node[coordinate] (x7) at (-0.05, 0.90) {};
    \node[coordinate] (y7) at (1.05, 0.80) {};
    \draw[line width=0.00mm, name path global/.expanded=xy7] (x7) -- (y7);
    
    \node[coordinate] (x8) at (-0.05, 0.80) {};
    \node[coordinate] (y8) at (1.05, 0.60) {};
    \draw[line width=0.00mm, name path global/.expanded=xy8] (x8) -- (y8);
    
    \node[coordinate] (x9) at (0.18, -0.05) {};
    \node[coordinate] (y9) at (0.42, 1.05) {};
    \draw[line width=0.00mm, name path global/.expanded=xy9] (x9) -- (y9);
    
    \node[coordinate] (x10) at (-0.05, 0.31) {};
    \node[coordinate] (y10) at (1.05, 0.24) {};
    \draw[line width=0.00mm, name path global/.expanded=xy10] (x10) -- (y10);
    
    \node[coordinate] (x11) at (-0.05, 0.72) {};
    \node[coordinate] (y11) at (1.05, 0.44) {};
    \draw[line width=0.00mm, name path global/.expanded=xy11] (x11) -- (y11);
    
    \node[coordinate] (x12) at (-0.05, 1.02) {};
    \node[coordinate] (y12) at (1.05, 0.01) {};
    \draw[line width=0.00mm, name path global/.expanded=xy12] (x12) -- (y12);
    
    \node[coordinate] (x13) at (0.72, 1.05) {};
    \node[coordinate] (y13) at (0.78, -0.05) {};
    \draw[line width=0.00mm, name path global/.expanded=xy13] (x13) -- (y13);

\end{tikzpicture}
		\caption[Network2]%
		{{\small Fiber network.}}    
		\label{fig:mean and std of net14}
	\end{subfigure}
	\hfill
	\begin{subfigure}[b]{0.475\textwidth}  
		\centering 
		\input{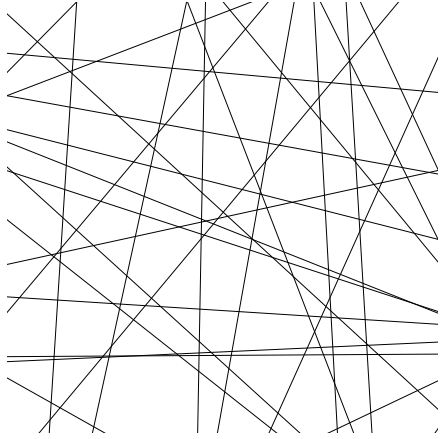}
		\caption[]%
		{{\small Domain $\Omega$ and nodes $\calN$ (red) introduced.}}    
		\label{fig:mean and std of net24}
	\end{subfigure}
	\vskip\baselineskip
	\begin{subfigure}[b]{\textwidth}   
		\centering 
		\begin{adjustwidth*}{0pt}{}
		\input{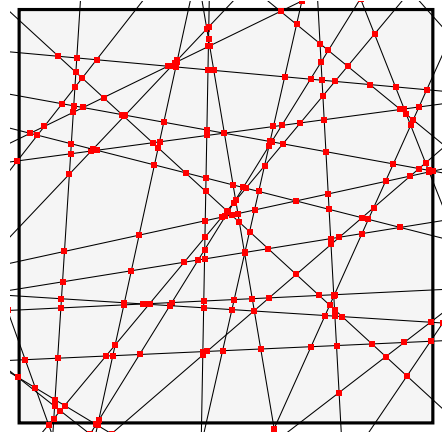}
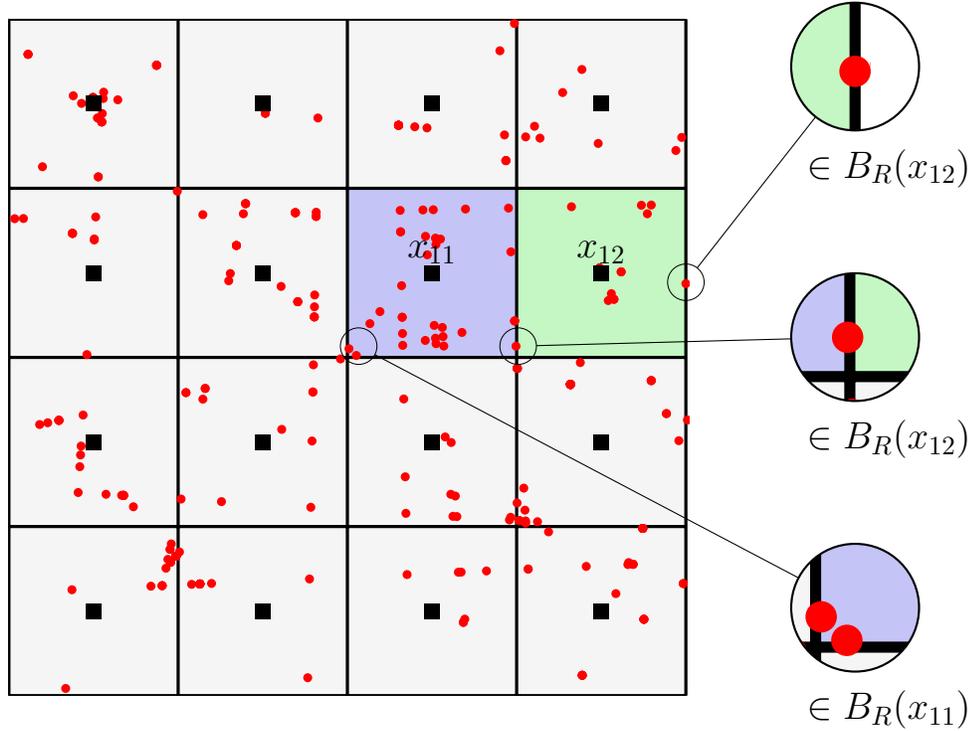
		\end{adjustwidth*}
		\caption[]%
		{{\small Grid coordinates $\{x_j\}$ inserted to define boxes, e.g., $B_R(x_{11})$ in blue, and $B_R(x_{12})$ in green.}}    
		\label{fig:mean and std of net34}
	\end{subfigure}
	\caption{Workflow of the spatial discretization. An example of fibers creating a network is seen in (A). In (B), a domain $\Omega$ has been defined in which the network is embedded, and the intersections of threads define the node set $\calN$, visualized in red. The boxes $B_R(x_j)$ are further defined by introducing grid coordinates ${x_j}$, depicted as black squares in (C). At last, (C) shows three examples of nodes placed on grid boundaries, and which corresponding box they belong to by the definition \eqref{eq:BR_definition}.}
	\label{fig:mean and std of nets}
\end{figure*}

Next, we introduce a finite element space of first order on the mesh $\TH$. Let $\hat{Q}_H$ be the space of continuous functions on $\Omega$ who, restricted to $T \in \TH$, are linear combinations of the polynomials $z=(z_1,\ldots,z_d) \mapsto z^\alpha$ for the multi-index $\alpha$ with $\alpha_i \in \{0,1\}$ for $i=1,\ldots,d$. That is, for, e.g., $d=2$, the space $\hat{Q}_H$ consists of bilinear functions on $\TH$. Furthermore, let
\begin{align}
	Q_H = \{q\in \hat{Q}_H : q\big|_{\Gamma}= 0\}
\end{align}
be the corresponding finite element space with incorporated boundary condition. Here we assume that the Dirichlet boundary condition is imposed on full element edges of the finite element mesh. We moreover wish to restrict these spaces to nodes in $\Omega$. Therefore, let $\hat{V}_H$ denote the space of functions in $\hat{Q}_H$ restricted to $\calN$, and likewise for $V_H$. 

We finish the spatial discretization by introducing a basis for the above defined spaces. Let $\{y_i\}_{i=1}^m$ denote the nodes of the mesh $\TH$. By $\{\phi_i\}_{i=1}^m$, we denote the standard Lagrange finite element basis functions that span $\hat{Q}_H$, and by $\{\varphi_i\}_{i=1}^m$ their restrictions to the network nodes, which in turn span $\hVH$. Moreover, we assume, without loss of generality, that the $m_0 < m$ first basis functions $\{\phi_i\}_{i=1}^{m_0}$ span $Q_H$, i.e., the space of functions with zero trace on $\Gamma$.

At this point, we can define the fully discretized problem corresponding to our model problem \eqref{eq:refmethod} by replacing the function space $V$ in \eqref{eq:discretemethodref} with the finite element space $\VH$. That is, one may seek a solution $\unH \in \VH$ such that
\begin{align}
		(M\fdiscd \bdiscd \unH, v) + (K\tfrac{1}{2}(\unHhp + \unHhm), v) = (\fn, v),\quad \forall v\in \VH.
\end{align}
However, an issue with this method comes with the fact that the matrix $K$ may contain highly varying and complex features inherited by the structure of the underlying network. These features can in turn not be resolved by the finite element space $\VH$, unless a sufficiently small mesh size $H$ is considered. If these network features vary on a scale of $\varepsilon$, we therefore require $H < \varepsilon$, which quickly becomes prohibitively small with decreasing $\varepsilon$. The main goal of this paper is to construct an efficient method by circumventing this issue, mainly by incorporating the features of $K$ into the finite element space, such that the features of the network can be accurately resolved even for coarse grid sizes $H > \varepsilon$. The framework, presented in the subsequent section, is based on the localized orthogonal decomposition method. It was first introduced in \cite{MalP14} and was later developed for network models in \cite{KetMMFWE20}.

\section{Localized orthogonal decomposition}\label{s:localized_orthogonal_decomposition}

This section is dedicated to the development of a localized orthogonal decomposition method for wave propagation on spatial networks. The main idea is to split the solution space into a coarse-scale part, onto which functions can be interpolated, and a fine-scale part that recovers the features not captured by the interpolant. The fine-scale space can in turn be utilized to incorporate the complex network features into the finite element space to construct a modified space, $\Vms$, with the same dimension as $\VH$, but with more accurate approximation properties.

We begin by introducing the interpolant that maps functions from the solution space $V$ onto the finite element space $\VH$. The definition is inspired by the construction of the Scott--Zhang interpolant, originally defined in \cite{ScoZ90}. More precisely, given a basis function $\varphi_j$, let the (unique) element containing $y_j$ be denoted by $T_j$, and define $\psi_j \in \hat{V}_H$ such that $(M_{T_j}\psi_j, \varphi_\ell) = \delta_{j\ell}$. The interpolant $\cI: V\rightarrow \VH$ is then defined as
\begin{align}
	\cI v := \sum_{j=1}^{m_0} (M_{T_j}\psi_j, v) \varphi_j.
	\label{eq:interpolant_definition}
\end{align}
In \cite{EdeGHKM22}, several useful results for $\cI$ are derived. In particular, the following lemma is presented, which is of great importance for the work in this paper.
\begin{lemma}[Interpolation error bound]\label{lem:interpolation_bound}
	If Assumption \ref{ass:network_properties} holds, and $H \geq R_0$, then the interpolant $\cI$ in \eqref{eq:interpolant_definition} satisfies
	\begin{align}
	H^{-1}\mnorm{v - \cI v} + \lnorm{\cI v} \leq C \lnorm{v}, \quad \forall v\in V,
	\end{align} 
	where the constant depends on the connectivity and uniformity of the network, as well as the dimension $d$, but is independent of the rapidly varying data inherited by the network.
\end{lemma}
Next, we let the kernel of $\cI$ define the fine-scale space, i.e.,
\begin{align}
	W := \text{ker}(\cI) = \{v\in V : \cI v = 0\},
\end{align}
such that the features not captured by $\cI$ are recovered in $W$. We now seek to use $W$ to compute these fine-scale network features, and in turn embed them into $\VH$. For this purpose, let $\calQ : V \rightarrow W$ be the so-called correction operator, such that $\calQ v\in W$ solves
\begin{align}
	(K\calQ v, w) = (Kv, w), \quad \forall w\in W.
	\label{eq:global_corrector_problem}
\end{align}
Our modified finite element space, $\Vms$, is now defined by subtracting the corrected space $\calQ \VH$, from the original space $\VH$, i.e.,
\begin{align}
	\Vms:= \VH - \calQ \VH = \{(I-\calQ)v : v\in \VH\}.
\end{align}
This construction yields the splitting $V = \Vms \oplus W$, such that elements from $\Vms$ and $W$ are orthogonal with respect to the $K$-inner product. Moreover, note that $\dim(\Vms) = \dim(\VH)$, such that the coarse dimension of $\VH$ is preserved, while information about the network and its fine-scale features are now taken into account in $\Vms$.

We are now ready to define an ideal version of our novel method. It reads: find $\unmsH \in \Vms$ such that
\begin{align}
	(M\fdiscd \bdiscd \unmsH, v) + (K\tfrac{1}{2}(\unmsHhp + \unmsHhm), v) = (\fn, v),\quad \forall v\in \Vms,
	\label{eq:ideal_method}
\end{align}
with $u^{\mathrm{ms}, 0}_H, u^{\mathrm{ms}, 1}_H$ being suitable approximations of $u(t_0)$ and $u(t_1)$ in $\Vms$.

The ideal method above manages to approximate the solution accurately, even for coarse grid sizes. The method however relies on the global corrector problem \eqref{eq:global_corrector_problem}, which in practice becomes prohibitively expensive in terms of computational complexity. Fortunately, each basis correction satisfies an exponential decay away from its node, and can therefore be computed on a local patch surrounding its node. Consequently, the complexity of computing a basis correction is significantly reduced, and its local support moreover makes the corresponding matrix system sparse.

For the localization, we begin by introducing coarse grid patches on the mesh $\TH$. Namely, for any subdomain $\omega \in \Omega$, we define
\begin{align}
U(\omega):= \{x\in \Omega: \exists T\in \TH : x\in T, \overline{T}\cap \overline{\omega} \neq \emptyset\}.
\label{eq:grid_patches}
\end{align}
That is, for, e.g., an element $T\in\TH$, the patch $U(T)$ contains all nodes in the element $T$, as well as the nodes in adjacent elements. Common choices for $\omega$ often include elements $T\in \TH$ or nodes $x\in \calN$, for which we write $U(x) := U(\{x\})$. Moreover, for $k\in \mathbb{N}$, we recursively define $U_k(\omega) := U_{k-1}(U(\omega))$, with $U_1 := U$. 

Given the coarse grid patches \eqref{eq:grid_patches}, we define the restricted fine-scale space
\begin{align}
	W^\omega_k := \{v\in W : \supp(v)\subseteq U_k(\omega)\},
\end{align}
for any subdomain $\omega \subseteq \Omega$. Next, we define a localized, element-restricted, corrector operator $\calQ^T_k : \VH \rightarrow W^T_k$ such that $\calQ^T_k v \in W^T_k$ solves
\begin{align}
	(K\calQ^T_k v, w) = (K_Tv, w), \quad \forall w\in W^T_k.
	\label{eq:element_correction}
\end{align}
In this way, the computation of $\calQ^T_k v$ is cheap for each element $T\in \TH$, due to the restricted domain $U_k(T)$ it is computed on. Moreover, we obtain the corresponding global version of the localized correction by summing over all $T\in \TH$, i.e.,
\begin{align}
	\calQ_k v = \sum_{T\in\TH} \calQ^T_k v,
\end{align}
since the matrix $K_T$ sums up to $K$. We are now ready to define the localized multiscale space in a similar fashion as $\Vms$, but replacing $\calQ$ by its localized version $\calQ_k$. That is, we have the space $\VmsHk := \VH - \calQ_k \VH$, spanned by the corrected basis $\{\varphi_i - \calQ_k \varphi_i\}_{i=1}^{m_0}$. For an illustration of a corrected basis function computed on a patch $U_k(x)$, see Figure~\ref{fig:mod}.

\begin{figure}
	\centering
	\includegraphics{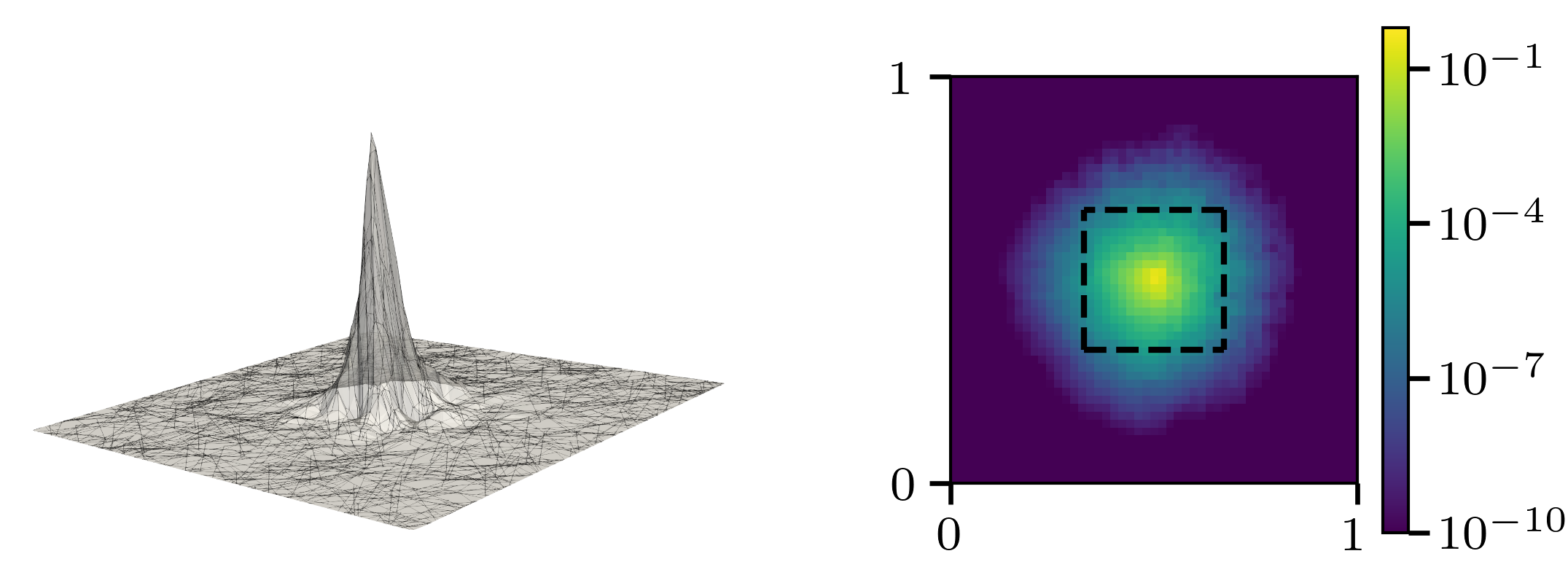}
	\caption{The left figure illustrates a typical corrected basis function, $\varphi_i - \calQ \varphi_i$, where $H=2^{-5}$, on $U_6(x_i)$ with $x_i = (0.5,0.5)$. To the right, we see the magnitude of the function on the entire domain. The dashed square in the right figure is $U_6(x_i)$.}
	\label{fig:mod}
\end{figure}

The localized version of \eqref{eq:ideal_method}, and the main method of this paper is stated as: find $\unmsHk \in \VmsHk$ such that
\begin{align}
	(M\fdiscd \bdiscd \unmsHk, v) + (K\tfrac{1}{2}(\unmsHkhp + \unmsHkhm), v) = (\fn, v),\quad \forall v\in \VmsHk,
	\label{eq:localized_method}
\end{align}
with initial values $u^{\mathrm{ms},0}_{H,k}, u^{\mathrm{ms},1}_{H,k} \in \VmsHk$ being suitable approximations of $u(t_0)$ and $u(t_1)$ from the reference method \eqref{eq:reference_equation_weak_form}.

We end the section by presenting following result, which shows the stability of the time-stepping scheme for our main method. The result follows by considering $\VmsHk$ as test function space in the proof of \cite[Lemma 13.2]{Larsson03}.
\begin{proposition}
	The solution of \eqref{eq:localized_method}, with source function $f=0$, satisfies
	\begin{align}
		\mnorm{\fdiscd \unmsHk}^2 + \knorm{\unmsHkhp}^2 = \mnorm{\fdiscd u^{\ms, 0}_{H,k}}^2 + \knorm{u^{\ms, 1/2}_{H,k}}^2, \quad n \geq 0.
	\end{align}
\end{proposition}

\section{Error bounds}\label{s:error_estimates}

In this section we derive error bounds for the localized method \eqref{eq:localized_method}. Before proving the main results, we introduce the necessary operators and their corresponding stability estimates required for the analysis.

\subsection{Operators and their estimates}

For the error analysis, we require a projection onto the space $\VH$ with respect to the $M$-inner product. Therefore, let $\PH: V\rightarrow \VH$ be defined such that $\PH v \in \VH$ solves
\begin{align}
	(M\PH v, w) = (Mv, w), \quad \forall w\in \VH.
	\label{eq:PH}
\end{align}
Moreover, we introduce the global Ritz-projection onto the multiscale space $\Vms$, i.e., the operator $\Rms : V\rightarrow \Vms$ such that $\Rms v\in \Vms$ solves
\begin{align}
(K\Rms v, w) = (Kv, w), \quad \forall w\in \Vms.
\end{align}
Similarly, we define its localized version $\Rmsk: V\rightarrow \VmsHk$ such that $\Rmsk v\in \VmsHk$ solves
\begin{align}
	(K\Rmsk v, w) = (Kv, w), \quad \forall w\in \VmsHk.
\end{align}
These operators, together with the localized corrector operator $\calQ_k$, play crucial roles in the error analysis. The following lemma on their stability will be frequently used.
\begin{lemma}\label{lem:stability_of_operators}
	The operators $\Rmsk$ and $\Qk$ are both stable with respect to $\knorm{\cdot}$, i.e., 
	\begin{align}
	\knorm{\Rmsk v} \leq \knorm{v}, \qquad
	\knorm{\Qk v} \leq \knorm{v}.
	\end{align}
\end{lemma}
\begin{proof}
	For $\Rmsk$, the stability follows straight-forwardly as
	\begin{align}
	\knorm{\Rmsk v}^2 = (K\Rmsk v, \Rmsk v) = (Kv, \Rmsk v) \leq \knorm{v} \knorm{\Rmsk v},
	\end{align}
	and similarly for $\Qk$.
\end{proof}

We further require stability of the operator $\PH$ in $K$-norm. To prove the stability, we first require an inverse inequality for shape-regular meshes, stated in the following lemma. In the following we use the notation $\lesssim$ to indicate that there may be a constant independent of the mesh size parameter $H$ that is not explicitly stated.

\begin{lemma}\label{lem:inverse_inequality}
	Any function $v\in \VH$ satisfies the inverse inequality
	\begin{align}
		\big| v \big|_{L,T} \leq \com{C_{d,\sigma}} H^{-1}\big| v \big|_{M,T}.
		\label{eq:inverse_inequality}
	\end{align}
\end{lemma}
\begin{proof}
	For $T\in \TH$, define $\hat{T} := \{x/H : x \in T\}$ as the hypercube scaled to unit size, and define for $v \in V_H$ the function $\hat{v}$ by 
	\begin{align}
	\hat{v}(\hat{x}) = v(H\hat{x}), \quad \hat{x} \in \hat{T}.
	\end{align}
	\com{We claim that the $\big|\cdot\big|_{L,\hat{T}}$-norm can be bounded by the $\big|\cdot\big|_{M,\hat{T}}$-norm as
	\begin{align}\label{eq:equivalence_of_norms}
		\big|v\big|_{L,\hat{T}} \leq C_{d,\sigma} \big|v\big|_{M,\hat{T}}, \quad \forall v\in \VH.
	\end{align}
	To show this, we first note that
	\begin{align}
		|v(x) - v(y)|  &= \Big| \sum_{i=1}^{2^d} \alpha_i (\varphi_i(x) - \varphi_i(y)) \Big| \\
		&\leq \underbrace{\Big( \sum_{i=1}^{2^d} \alpha_i^2 \Big)^{1/2}}_{=:|\alpha|} \Big( \sum_{i=1}^{2^d} (\varphi_i(x) - \varphi_i(y))^2 \Big)^{1/2} \\
		&\leq |\alpha| |x-y| 2^{d/2}\sqrt{d},
	\end{align}
	where in the final step we have used the Lipschitz continuity of $\varphi_i(\cdot)$ with Lipschitz constant $\sqrt{d}$. Moreover, denote by $\Xi_{ij} = (M_{\hat{T}}\varphi_i, \varphi_j)$ and we note that
	\begin{align}
		\big|v\big|^2_{M,\hat{T}} = (M_{\hat{T}}v,v) = \alpha^T \Xi \alpha \geq \lambda_1(\Xi)|\alpha|^2 \geq C_d \rho |\alpha|^2,
	\end{align}
	where $\lambda_1(\Xi)$ denotes the smallest eigenvalue of $\Xi$, whose bound used in the last step follows in the same way as in the proof of \cite[Lemma 3.4]{EdeGHKM22}. The claim now follows, since
	\begin{align}
		\big|v\big|_{L,\hat{T}}^2 &= \sum_{x\in \hat{T}} \frac{1}{2}\sum_{y\sim x} \frac{(v(x) - v(y))^2}{|x-y|} \leq \sum_{x\in \hat{T}} \frac{1}{2} \sum_{y\sim x} |\alpha|^2 |x-y| 2^{d}d \\
		&\leq 2^{d}d |\alpha|^2 \big| 1 \big|_{M,\hat{T}} \leq 2^{2d}d|\alpha|^2 \sigma \rho \leq C_d\sigma \big|v\big|_{M,\hat{T}}^2.
	\end{align}}
	Now, it remains to find the relation between the norms of $\hat{v}$ on $\hat{T}$ and $v$ on $T$. We first note that
	\begin{align*}
	\big| v \big|^2_{M,T} &= \frac{1}{2} \sum_{x\in \calN(T)} \sum_{y\sim x} |x - y| v(x)^2 = \frac{1}{2} \sum_{\hat{x}\in \calN(\hat{T})} \sum_{\hat{y}\sim \hat{x}} H|\hat{x} - \hat{y}| v(H\hat{x})^2 = H\big| \hat{v}\big|^2_{M,\hat{T}},
	\end{align*}
	where the substitution $x = H\hat{x}$ was applied. Similarly, we see for the remaining norm that
	\begin{align*}
	\big| v\big|^2_{L,T} &= \frac{1}{2} \sum_{x\in \calN(T)} \sum_{y\sim x} \frac{(v(x) - v(y))^2}{|x - y|} = \frac{1}{2} \sum_{\hat{x}\in \calN(\hat{T})} \sum_{\hat{y}\sim \hat{x}} \frac{(v(H\hat{x}) - v(H\hat{y}))^2}{H|\hat{x} - \hat{y}|} = H^{-1}\big| \hat{v} \big|^2_{L,\hat{T}}.
	\end{align*}
	Inserting these expressions into \eqref{eq:equivalence_of_norms} yields 
	\begin{align}
	\big| v \big|_{L,T} \leq \com{C_{d,\sigma}}H^{-1}\big| v \big|_{M,T}.
	\end{align}
\end{proof}

Using Lemma \ref{lem:inverse_inequality}, we may furthermore show the stability of $\PH$.
\begin{lemma}\label{lem:PH_stability}
	The operator $\PH$ defined in \eqref{eq:PH} satisfies the stability estimate
	\begin{align}
		\knorm{\PH v} \lesssim \knorm{v}, \quad \forall v\in V.
	\end{align}
\end{lemma}
\begin{proof}
	First of all, note that using \eqref{eq:inverse_inequality}, we have for $v\in \VH$
	\begin{align}
		\lnorm{v}^2 = \sum_{T\in \TH} \big| v \big|^2_{L,T} \leq C\sum_{T\in \TH} H^{-2} \big| v \big|^2_{M,T} \leq C H^{-2} \mnorm{v}^2,
	\end{align}
	where the constant $C$ depends on the dimension $d$. The stability of $\PH$ in $L$-norm then follows, since
	\begin{align}
		\lnorm{\PH v} &= \lnorm{\PH(v-\cI v) + \cI v} \\
		&\leq \lnorm{\PH(v-\cI v)} + \lnorm{\cI v} \\
		&\leq CH^{-1}\mnorm{\PH(v-\cI v)} + \lnorm{\cI v} \\
		&= CH^{-1}\mnorm{v-\cI v} + \lnorm{\cI v} \\
		&\leq C\lnorm{v},
	\end{align}
	where the last inequality follows from the interpolation estimate in  Lemma~\ref{lem:interpolation_bound}. By \eqref{eq:KL_relation}, stability in $K$-norm follows as well.
\end{proof}

We continue by stating the necessary exponential decay that the localized corrector operator satisfy, in the following lemma, proven in \cite[Lemma 4.7]{EdeGHKM22}. The proof utilizes the same assumptions made in this paper, and thus the result follows accordingly.
\begin{lemma}\label{lem:decay_of_localized_corrector}
	The error between the global corrector operator $\calQ$ and its localized counterpart $\Qk$, satisfies
	\begin{align}
	\knorm{(\calQ-\Qk)v} \leq C\com{k^{d/2}} \exp(-ck)\knorm{v}
	\end{align}
	for $v\in V$, where the constants are independent of the variations in the data inherited by the network model.
\end{lemma}

\begin{remark}
	\com{The constant $c$ is proportional to $\sigma^{-1}\mu^{-2}\alpha\beta^{-1}$, see \cite[Section 3]{EdeGHKM22} and \cite[Section 5]{GorHM22} for the details. The Poincar\'{e} type constant $\mu$ is
	 investigated for networks used in the numerical examples of this paper in \cite[Section 3.1]{EdeGHKM22}. It is of moderate size unless the distribution of lines is sparse compared to the scale $H$. Since we aim for an overall error of order $H$ we will henceforth pick $k\sim \log(1/H)$. Note that the proportionality constant in this relation depends on $1/c\sim \mu^2\sigma\beta/\alpha$. The constant $C$ depends on $\sqrt{\beta/\alpha}$, see  proof of \cite[Lemma 4.7]{EdeGHKM22}. 
	
	We see the same effect when LOD is applied to elliptic PDEs (see, e.g., the original LOD paper \cite{MalP14}). The localization is negatively affected by high contrast in the diffusion. However, in practice, the localization tend to work well even in high contrast regimes, see the numerical examples sections in, e.g., \cite{MalP14, MalP18} and the final numerical example in this paper.}
\end{remark}

We finish the section on operators by deriving necessary estimates for the localized Ritz-projection $\Rmsk$. For this, we first require an estimate for the global operator $\Rms$, deduced in the following lemma. The proof is based on standard arguments for Ritz-projections, and is done in similarity with \cite[Lemma 3.2]{MalP18}, replacing the energy norm by the $K$-norm.

\begin{lemma}\label{lem:rho_estimate}
	For a solution $u\in V$ to \eqref{eq:refmethod}, it holds that
	\begin{align}
	\knorm{u - \Rms u} &\leq \frac{CH}{\sqrt{\alpha}} \mnorm{f- \ddu}.
	\label{eq:solution_rho_estimate_in_k_norm}
	\end{align}
\end{lemma}

The subsequent lemma shows similar estimates for the localized Ritz-projection. In the proof, the exponentially decaying factor from Lemma~\ref{lem:decay_of_localized_corrector} appears. However, we emphasize that choosing the localization parameter as $k\sim \log(1/H)$ makes the exponential factor proportional to $H$. We show two different results, where the first part holds for any arbitrary function $v\in V$ and requires less regularity, while the second holds for any solution $u\in V$ to \eqref{eq:reference_equation_weak_form}. The derived results are modifications of the result obtained in \cite[Lemma 5.3]{AbdH17}.

\begin{lemma}\label{lem:localized_ritz_projection_bounds}
	Let the localization parameter be chosen such that $k\sim \log(1/H)$. Then, for any function $v\in V$, the localized Ritz-projection satisfies the estimate
	\begin{align}
		\mnorm{\Rmsk v - v} \lesssim H\knorm{\Rmsk v - v} \lesssim H\knorm{v}. \label{eq:rmsk_mnorm_v}
	\end{align}
	Moreover, for the solution $u\in V$ to \eqref{eq:reference_equation_weak_form} and its derivatives, we have for $i=0,1,2,\ldots$
	\begin{align}
	\knorm{\Rmsk(D^i_t u) - D^i_t u} &\lesssim H\big(\mnorm{D^i_t f - D^{i+2}_tu} + \knorm{D^i_tu}\big).\label{eq:rmsk_knorm_u}
	\end{align}
\end{lemma}
\begin{proof}
	We begin by deriving the estimate \eqref{eq:rmsk_knorm_u}. Define the functions $v := D^i_t u$ and $\hf = D^i_t f$, and observe that $v$ satisfies the equation
	\begin{align}
	(MD^2_tv, w) + (Kv, w) = (M\hf, w), \quad \forall w\in V.
	\end{align}
	Next, analogously to the proof of \cite[Lemma 5.3]{AbdH17}, define the energy
	\begin{align}
	E(\phi_H) := (K(v-\phi_H+\Qk \phi_H), v-\phi_H+\Qk \phi_H), \quad \phi_H\in \VH.
	\label{eq:energy_function}
	\end{align}
	We write the decomposition $\Rmsk v = \vHk - \Qk \vHk$, and claim that $\vHk \in \VH$ minimizes the energy $E(\cdot)$. This can be seen, since first of all
	\begin{align}
	(K(v-\Rmsk v), \phi_H - \Qk \phi_H) = 0, \quad \forall v\in \VH,
	\end{align}
	by the definition of $\Rmsk$. Therefore, by perturbing $\vHk$ with arbitrary $\eta \in \VH$, and denoting $\eta^{\mathrm{ms}}_k := \eta - \Qk \eta$, we see that
	\begin{align}
	E(\vHk + \eta) &= (K(v - \vHk + \Qk \vHk - \eta + \Qk \eta),v - \vHk + \Qk \vHk - \eta + \Qk \eta) \\
	&= (K(v-\Rmsk v - \eta^{\mathrm{ms}}_k), v-\Rmsk v - \eta^{\mathrm{ms}}_k)= (K(v-\Rmsk v), v) + \knorm{\eta^{\mathrm{ms}}_k}^2 \\
	&\geq (K(v-\Rmsk v), v) = (K(v-\Rmsk v), v - \Rmsk v) =  E(\vHk),
	\end{align}
	which shows the claim. Consequently, we can estimate
	\begin{align}
		\begin{split}
		\knorm{\Rmsk v - v} &= \knorm{\vHk - \Qk \vHk - v} \\
		&\leq \knorm{\vH - \Qk \vH - v} \\
		&= \knorm{\vH - \calQ\vH + \calQ\vH - \Qk \vH - v} \\
		&\leq \knorm{\Rms v - v} + \knorm{(\calQ - \Qk)\vH} \\
		&\leq \frac{CH}{\sqrt{\alpha}}\mnorm{\hf - D^2_t v} + C\com{k^{d/2}} \exp(-ck)\knorm{\vH},
		\end{split}
		\label{eq:expr}
	\end{align}
	where in the last step we applied \eqref{eq:solution_rho_estimate_in_k_norm} from Lemma \ref{lem:rho_estimate} for the first term, and Lemma \ref{lem:decay_of_localized_corrector} for the second. Furthermore, with $v = \vH + w$ for $w\in W$, we may apply the following estimate 
	\begin{align}
	\knorm{\vH} = \knorm{\PH \vH} = \knorm{\PH(\vH + w)}  = \knorm{\PH v} \lesssim \knorm{v},
	\label{eq:vh_to_v_bound}
	\end{align}
	where we used Lemma~\ref{lem:PH_stability} for the last inequality. Combine \eqref{eq:expr} with \eqref{eq:vh_to_v_bound}, and we have
	\begin{align}
		\knorm{\Rmsk v - v} \lesssim \frac{CH}{\sqrt{\alpha}}\mnorm{\hf - D^2_t v} + C \com{k^{d/2}}\exp(-ck)\knorm{v}.
	\end{align}
	Choose $k \sim \log(1/H)$ and we deduce \eqref{eq:rmsk_knorm_u}.
	
	For the estimates in $M$-norm, define the error $\rhok := \Rmsk v - v$, and consider the dual problem to find $z\in V$ such that
	\begin{align}
	(Kw, z) = (Mw, \rhok), \quad \forall w\in V,
	\end{align}
	and the similar dual problem to find $\zmsHk \in \VmsHk$ such that
	\begin{align}
	(Kw, \zmsHk) = (Mw, \rhok), \quad \forall w\in \VmsHk.
	\end{align}
	Subtracting the equations, we note that the Galerkin orthogonality $(Kw,z-\zmsHk) = 0$ holds for $w\in \VmsHk$. Likewise, we have the orthogonality $(Kw, z-\zms)=  0$ in the non-localized space $\Vms$, which implies that $z-\zms \in W = \ker(\cI)$. Using this, we first note that
	\begin{align}
	(K(z-\zms), z-\zms) &= (K(z-\zms), z) \\
	&= (M(z-\zms), \rhok) \\
	&= (M(z-\zms - \cI(z-\zms)), \rhok) \\
	&\leq CH \knorm{z-\zms} \mnorm{\rhok} .
	\end{align}
	Next, write $\zmsHk = \zHk - \Qk\zHk$ and note that it minimizes the energy $E(\cdot)$ in \eqref{eq:energy_function}, so in a similar fashion as before we get
	\begin{align}
	\knorm{z-\zmsHk} &= \knorm{z-\zHk + \Qk \zHk} \\
	&\lesssim \knorm{z-\zH + \Qk \zH} \\
	&= \knorm{z - \zH + Q\zH - Q\zH + \Qk \zH} \\
	&\lesssim \knorm{z-\zms} + \knorm{(Q-\Qk)\zH} \\
	&\lesssim H\mnorm{\rhok} + C\com{k^{d/2}}\exp(-ck)\knorm{\zH}.
	\end{align}
	Moreover, it holds that
	\begin{align}
	\knorm{\zH} =\knorm{\PH \zH} = \knorm{\PH(\zH - \Qk \zH)} = \knorm{\PH \zms} \lesssim \knorm{\zms} \lesssim \mnorm{\rhok},
	\end{align}
	where the last two inequalities follow from stability of $\PH$ in $K$-norm, and the fact that $\knorm{z} \lesssim \mnorm{\rhok}$ and $\knorm{z - \zms} \lesssim H\mnorm{\rhok}$. Now, we can deduce
	\begin{align}
	\mnorm{\rhok}^2 = (K\rhok, z) = (K\rhok, z-\zmsHk) \lesssim \knorm{\rhok}(H+C\com{k^{d/2}}\exp(-ck))\mnorm{\rhok},
	\label{eq:expr2}
	\end{align}
	where we used the dual problem for the first equality, and the second follows since $(K\Rmsk v, \zmsHk) = (Kv, \zmsHk)$, and thus $(K\rhok, \zmsHk) = 0$.  By choosing $k \sim \log(1/H)$ in \eqref{eq:expr2}, we get \eqref{eq:rmsk_mnorm_v}, which concludes the proof.
\end{proof}

\subsection{Error bounds}

This section is dedicated to the derivations of the main result of this paper, namely the optimal order convergence of the novel method stated in \eqref{eq:localized_method}. As a measure of error, we compare the approximate solution $\unmsHk$ to the reference solution $u(\tn)$ in $K$-norm, as well as the derivative $\fdiscd \unmsHk$ versus $D_t u(\tn)$ in $M$-norm, which is analogous to the standard measurement in the continuous setting for the linear wave equation. For convenience, we begin by deriving the error in $K$-norm in Theorem~\ref{lem:error_k_norm}, and then the error in $M$-norm in Theorem~\ref{lem:error_m_norm}, separately. We then summarize the main result as their combination, where we have additionally added certain assumptions, such as choice of initial data and its order of compatibility and well-preparedness, to emphasize the convergence order of the method.

\begin{theorem}\label{lem:error_k_norm}
	Let the localization parameter be chosen such that $k\sim \log(1/H)$. Then, the error between the solution $u(t)\in V$ to \eqref{eq:reference_equation_weak_form} and the solution $\unmsHk \in \VmsHk$ to \eqref{eq:localized_method} satisfies
	\begin{align}
		\begin{split}
		\max_{1\leq n\leq N-1}\knorm{\unmsHkhp - &u(\tnhp)} \\
		&\lesssim H\big( \linfmnorm{f} + \linfmnorm{D^2_t u} + \linfknorm{u} + T \linfknorm{D^2_t u} \big) \\
		& \qquad + \tau^2 \big(\lonemnorm{D^4_tu} + \lonemnorm{D^2_t f} + H\loneknorm{D^2_tu} \big) \\
		& \qquad + \mnorm{\fdiscd u^{\mathrm{ms},0}_{H,k} - \Rmsk \fdiscd u(t_0)} + \knorm{u^{\mathrm{ms},1/2}_{H,k} - \Rmsk u(\tohp)}.
		\end{split}\label{eq:error_k_norm}
	\end{align}
	
\end{theorem}
\begin{proof}
	Define the error as
	\begin{align}
	e^n := \unmsHk - u(\tn) = \unmsHk - \Rmsk u(\tn) + \Rmsk u(\tn) - u(\tn) = \thnk + \rhonk.
	\end{align}
	For $\rhonk$, we know from Lemma~\ref{lem:localized_ritz_projection_bounds} that
	\begin{align}
	\knorm{\rhonkhp} &= \knorm{\Rmsk u(\tnhp) - u(\tnhp)} \\
	&\lesssim H\big(\mnorm{f(\tnhp) - D^2_t u(\tnhp)} + \knorm{u(\tnhp)}\big).
	\end{align}
	For $\thnk$, we note that it satisfies the equation
	\begin{align}
	(M\fdiscd \bdiscd \thnk, &v) + (K\tfrac{1}{2}(\thnkhp + \thnkhm), v) \\
	&= (M\fdiscd \bdiscd \unmsHk, v) - (M\fdiscd \bdiscd \Rmsk u(\tn), v) + (K\tfrac{1}{2}(\unmsHkhp + \unmsHkhm), v)  \\
	& \qquad \qquad - (K\tfrac{1}{2}(u(\tnhp) + u(\tnhm)), v) \\
	&= (M\fn, v) - (M\fdiscd \bdiscd \Rmsk u(\tn), v) - (Ku(\tn), v) - (K\tfrac{1}{2}\Qnddu, v) \\
	&= (M\ddu, v) -  (M\fdiscd \bdiscd \Rmsk u(\tn), v) + (M\tfrac{1}{2}\Qndddduddf, v) \\
	&= (M\omni, v) + (M\omnii, v) + (M\tfrac{1}{2}\Qndddduddf, v)
	\end{align}
	where $\omni:= D^2_t{u}(t_n) - \fdiscd \bdiscd u(\tn)$, $\omnii:= \fdiscd\bdiscd u(\tn) - \fdiscd \bdiscd \Rmsk u(\tn)$, and $\Theta$ is the function defined in \eqref{eq:Theta}. Therefore, by Lemma~\ref{lem:discreteschemeestimate}, we have that 
	\begin{align}
	\mnorm{\fdiscd \thnk} + \knorm{\thnkhp} \leq C\Big( \mnorm{\fdiscd \thok} + \knorm{\thkhp} + \sumin \tau \big(\mnorm{\omii} + \mnorm{\omiii} + \mnorm{\Qidddduddf}\big)\Big).
	\end{align}
	At first, we note that by Taylor expansion, we have
	\begin{align}
	\mnorm{\omii} = \mnorm{\fdiscd \bdiscd u(\ti) - \ddu(\ti)} = \intmnorm{\frac{1}{6\tau^2} \int_{\tiprev}^{\tinext} \ddddu(s)\Lambda_{[\tiprev,\tinext]}^3(s)\, \ds} \leq \frac{\tau}{6}\int_{\tiprev}^{\tinext} \mnorm{\ddddu(s)}\, \ds,
	\end{align}
	with $\Lambda_{[\tiprev,\tinext]}$ as defined in \eqref{eq:lambda_function}, so that the first sum can be bounded as
	\begin{align}
	\sumin \tau \mnorm{\omii} \leq \frac{\tau^2}{6}\sumin \int_{\tiprev}^{\tinext} \mnorm{\ddddu(s)}\, \ds \leq \frac{\tau^2}{3} \int_0^{\tnnext} \mnorm{\ddddu(s)}\, \ds.
	\end{align}
	Moreover, we can write
	\begin{align}
	\mnorm{\omiii} &= \mnorm{\fdiscd \bdiscd u(\ti) - \fdiscd \bdiscd \Rmsk u(\ti)} = \mnorm{(I - \Rmsk)\fdiscd \bdiscd u(\ti)} \\
	&= \mnorm{(I-\Rmsk)(\ddu(\ti) +  \Qiddu)} \leq \mnorm{D^2_t \rhoik} + \mnorm{(I-\Rmsk)\Qiddu},
	\end{align}
	where we can bound the corresponding first sum using Lemma~\ref{lem:localized_ritz_projection_bounds} as
	\begin{align}
	\sumin \tau \mnorm{D^2_t \rhoik} \lesssim H \sumin\tau \knorm{D^2_t u(\ti)}
	\end{align}
	and the second as
	\begin{align}
	\sumin \tau \mnorm{(I-\Rmsk)\Qiddu} &= \sumin \tau \intmnorm{(I-\Rmsk)\int_{\tihm}^{\tihp} \ddu(s)\Lambda_{[\tihm,\tihp]}(s)\, \ds} \\
	&\leq \sumin \frac{\tau^2}{2} \intmnorm{\int_{\tihm}^{\tihp} (I-\Rmsk)\ddu(s)\, \ds} \\
	&\leq \frac{\tau^2}{2} \sumin  \int_{\tihm}^{\tihp} \mnorm{\ddu(s) - \Rmsk\ddu(s)}\, \ds \\
	&\lesssim \tau^2H \int_{\tohp}^{\tnhp} \knorm{D^2_t u(s)}\, \ds,
	\end{align}
	where $\Lambda$ is the function defined in \eqref{eq:lambda_function}. The last part is bounded straight-forwardly as
	\begin{align}
	\sumin \tau \mnorm{\Qidddduddf} &= \sumin \tau \intmnorm{\int_{\tihm}^{\tihp} (\ddddu(s) - \ddf(s))\Lambda_{[\tihm,\tihp]}(s)\, \ds} \\
	&\leq \frac{\tau^2}{2}\sumin \int_{\tihm}^{\tihp} \mnorm{\ddddu(s) - \ddf(s)}\, \ds \\
	&\leq \frac{\tau^2}{2} \int_{\tohp}^{\tnhp} \mnorm{\ddddu(s) - \ddf(s)}\, \ds.
	\end{align}
	In total, $\thn$ satisfies the estimate
	\begin{align}
	\mnorm{\fdiscd \thnk} + \knorm{\thnkhp} &\leq C\Big( \mnorm{\fdiscd \thok} + \knorm{\thkhp} + \frac{\tau^2}{3} \int_0^{\tnnext} \mnorm{\ddddu(s)}\, \ds \\
	& \qquad \qquad \qquad + H \sumin\tau \knorm{D^2_t u(\ti)} \\
	& \qquad \qquad \qquad + \tau^2H \int_{\tohp}^{\tnhp} \knorm{D^2_t u(s)}\, \ds   \\
	& \qquad \qquad \qquad +\frac{\tau^2}{2} \int_{\tohp}^{\tnhp} \mnorm{\ddddu(s) - \ddf(s)}\, \ds\Big).
	\end{align}
	We combine the estimates for $\rhonk$ and $\thnk$, and obtain the total error
	\begin{align}
		\knorm{\enhp_k} &\lesssim H\Big( \mnorm{f(\tnhp)} + \mnorm{D^2_t u(\tnhp)} + \knorm{u(\tnhp)} +  \sumin\tau \knorm{D^2_t u(\ti)} \Big) \\
		& \qquad + \tau^2 \Big( \int_0^{\tnnext} \mnorm{D^4_tu(s)} + \mnorm{D^2_t f(s)} + H\knorm{D^2_t u(s)}\, \ds \Big) \\
		& \qquad + \mnorm{\fdiscd \thok} + \knorm{\thkhp}.
	\end{align}
	Now take the maximum over $n=1,\ldots, N-1$, and we arrive at \eqref{eq:error_k_norm}.
\end{proof}

\begin{theorem}\label{lem:error_m_norm}
	Let the localization parameter be chosen such that $k\sim \log(1/H)$. Then, the error between the solution $\unmsHk \in \VmsHk$ to \eqref{eq:localized_method} and $u\in V$ to \eqref{eq:reference_equation_weak_form} satisfies
	\begin{align}
	\begin{split}
	\max_{1\leq n\leq N-1}\mnorm{\fdiscd \unmsHk - &D_t{u}(\tnhp)} \\
	&\lesssim \mnorm{\fdiscd u^{\mathrm{ms},0}_{H,k} - \Rmsk \fdiscd u(t_0)} + \knorm{u^{\mathrm{ms},1/2}_{H,k} - \Rmsk u(\tohp)} \\
	& \quad + \tau^2 \big( \lonemnorm{D^4_tu} + \lonemnorm{\ddf} + H\loneknorm{D^2_tu} + \linfmnorm{\dddu} \big) \\
	&\quad + H\big( T\linfknorm{D^2_t u} +  \linfknorm{D_tu} \big)
	\end{split}\label{eq:error_m_norm}
	\end{align}
\end{theorem}
\begin{proof}
	We can split the error as
	\begin{align}
	\fdiscd \unmsHk - D_t{u}(\tnhp) &= \fdiscd \unmsHk - \fdiscd \Rmsk u(\tn) + \fdiscd \Rmsk u(\tn) - \fdiscd u(\tn) + \fdiscd u(\tn) - D_t{u}(\tnhp) \\
	&= \fdiscd \thnk + \fdiscd \rhonk + \fdiscd u(\tn) - D_t{u}(\tnhp).
	\end{align}
	From the proof of Lemma~\ref{lem:error_k_norm}, we already know that
	\begin{align}
	\mnorm{\fdiscd \thnk}  &\lesssim  \mnorm{\fdiscd \thok} + \knorm{\thkhp} + H \sumin\tau \knorm{D^2_t u(\ti)}  \\
	& \qquad  + \tau^2\Big( \int_0^{\tnnext} \mnorm{\ddddu(s)} + \mnorm{\ddf(s)} + H\knorm{D^2_t u(s)}\, \ds \Big) .
	\end{align}
	For $\rhonk$, we first note that
	\begin{align}
	\fdiscd u(\tn) = D_t{u}(\tn) + \frac{1}{\tau}\int_{\tn}^{\tnnext} \ddu(s)(\tnnext-s)\, \ds.
	\end{align}
	Therefore we have, using Lemma \ref{lem:localized_ritz_projection_bounds} with $k\sim \log(1/H)$
	\begin{align}
	\mnorm{\fdiscd \rhonk} &= \mnorm{(I-\Rmsk)\fdiscd u(\tn)} \\
	&\leq \mnorm{D_t{u}(\tn) - \Rmsk D_t{u}(\tn)} + \frac{1}{\tau}\intmnorm{(I-\Rmsk)\int_{\tn}^{\tnnext} \ddu(s)(\tnnext-s)\, \ds} \\
	&\leq \mnorm{D_t{u}(\tn)-\Rmsk D_t{u}(\tn)} + \int_{\tn}^{\tnnext} \mnorm{\ddu(s) -\Rmsk \ddu(s)}\, \ds \\
	&\lesssim H\Big( \knorm{D_t u(\tn)} + \int_{\tn}^{\tnnext} \knorm{D^2_tu(s)}\, \ds \Big) \\
	&\leq  H\Big( \knorm{D_t u(\tn)} + \tau \max_{[\tn,\tnnext]} \knorm{D^2_tu(s)} \Big).
	\end{align}
	For the last part, we can once again use Taylor expansion to get
	\begin{align}
	\mnorm{\fdiscd u(\tn) - D_t{u}(\tnhp)} &= \intmnorm{\frac{1}{2\tau} \int_{\tn}^{\tnnext} \dddu(s)\Lambda_{[\tn,\tnnext]}^2(s)\, \ds} \\
	&\leq \frac{\tau}{8}\int_{\tn}^{\tnnext} \mnorm{\dddu(s)}\, \ds \leq \frac{\tau^2}{8}\max_{[\tn,\tnnext]}\mnorm{\dddu(s)}.
	\end{align}
	Together we have the estimate
	\begin{align}
	\mnorm{\fdiscd \unmsHk - D_t{u}(\tnhp)} &\lesssim \mnorm{\fdiscd \thok} + \knorm{\thkhp} + \tau^2 \max_{[\tn,\tnnext]} \mnorm{\dddu(s)} \\
	& \qquad + \tau^2 \Big( \int_0^{\tnnext} \mnorm{D^4_tu(s)} + \mnorm{\ddf(s)} + H\knorm{D^2_tu(s)}\,\ds \Big) \\
	&\qquad + H\Big( \sumin \tau \knorm{D^2_t u(\ti)} + \tau \max_{[\tn,\tnnext]} \knorm{D^2_t u(s)} + \knorm{D_tu(s)} \Big).
	\end{align}
	Take the maximum over $n=1,\ldots,N-1$ and we obtain \eqref{eq:error_m_norm}.
\end{proof}
We finish this section by proving the main result of this paper, which follows by a combination of Theorem~\ref{lem:error_k_norm} and Theorem~\ref{lem:error_m_norm}, together with additional assumptions on the initial data. 
\begin{corollary} \label{cor:comp_conv_rate}
	Let the localization parameter be chosen such that $k\sim \log(1/H)$, and assume the initial values of the proposed method, $u^{\mathrm{ms},0}_{H,k}, u^{\mathrm{ms},1}_{H,k}$, are chosen such that 
	\begin{align}
		\mnorm{\fdiscd u^{\mathrm{ms},0}_{H,k} - \Rmsk \fdiscd u(t_0)} + \knorm{u^{\mathrm{ms},1/2}_{H,k} - \Rmsk u(\tohp)} \leq C(H + \tau^2).
		\label{eq:assumption_on_initial_data}
	\end{align}
	Moreover, assume the data is well-prepared and compatible of order 3 in the sense of Definition \ref{def:well_prepared}. Then, the error between the approximate solution $\unmsHk$ to \eqref{eq:localized_method} and the reference solution $u(\tn)$ to \eqref{eq:reference_equation_weak_form}, satisfies
	\begin{align}
		\mnorm{\fdiscd \unmsHk - D_t{u}(\tnhp)} + \knorm{\unmsHkhp - u(\tnhp)} \leq C(H + \tau^2),
	\end{align}
	where the constant $C$ is independent of the complex features inherited by the network structure.
\end{corollary}
\begin{proof}
	Combining the results from Theorem~\ref{lem:error_k_norm} and Theorem~\ref{lem:error_m_norm}, we have the full expression
	\begin{align}
	\begin{split}
		&\max_{1\leq n\leq N-1} \big( \mnorm{\fdiscd \unmsHk - D_t{u}(\tnhp)} + \knorm{\unmsHkhp - u(\tnhp)} \big) \\
		&\qquad\lesssim 	\mnorm{\fdiscd u^{\mathrm{ms},0}_{H,k} - \Rmsk \fdiscd u(t_0)} + \knorm{u^{\mathrm{ms},1/2}_{H,k} - \Rmsk u(\tohp)} \\
		&\qquad \qquad + \tau^2\big(  \lonemnorm{D^4_tu} + \lonemnorm{\ddf} + H\loneknorm{D^2_tu} + \linfmnorm{\dddu} \big) \\
		&\qquad \qquad + H\big( \linfmnorm{f} + \linfmnorm{D^2_t u} + \linfknorm{u} + (T+\tau)\linfknorm{D^2_t u} + \linfknorm{D_tu}  \big).
	\end{split}\label{eq:error_expr}
	\end{align}
	The first two terms in the right-hand side expression are immediately bounded by the assumption \eqref{eq:assumption_on_initial_data}. For the remaining terms, we will rely on the regularity result from Lemma~\ref{lem:regularity_result}, together with the well-preparedness and compatibility of the data. First of all, the result in Lemma~\ref{lem:regularity_result} can be extended to the Bochner space norms as
	\begin{align}
			\lonemnorm{D^{m+1}_t{u}} + \loneknorm{D^{m}_tu} &\lesssim T\big(\mnorm{w_{m+1}} + \knorm{w_m} + \lonemnorm{D^{m}_tf}\big), \\
			\linfmnorm{D^{m+1}_t{u}} + \linfknorm{D^{m}_tu} &\lesssim \mnorm{w_{m+1}} + \knorm{w_m} + \lonemnorm{D^{m}_tf}.
	\end{align}
	By applying this result to all relevant terms in \eqref{eq:error_expr}, together with the bound from the initial data assumption, we get the simplified error estimate
	\begin{align}
		&\max_{1\leq n\leq N-1} \big( \mnorm{\fdiscd \unmsHk - D_t{u}(\tnhp)} + \knorm{\unmsHkhp - u(\tnhp)} \big) \\
		&\qquad\lesssim_T C(H+\tau^2) + H\Big( \linfmnorm{f} + \sum_{j=0}^2 \big( \mnorm{w_{j+1}} + \knorm{w_j} + \lonemnorm{D^j_t f} \big) \Big)  \\
		&\qquad \qquad + \tau^2\sum_{j=2}^3\big( \mnorm{w_{j+1}} + \knorm{w_j} + \lonemnorm{D^j_t f} \big).
	\end{align}
	Finally, by the assumption that the data is well-prepared and compatible of order 3, the sought result follows by Definition~\ref{def:well_prepared}.
\end{proof}

\begin{remark}[Choice of initial data] \label{remark:good_initial}
	The assumption made in \eqref{eq:initial_data} on the convergence of the initial data can be obtained by, e.g., choosing $u^{\mathrm{ms}, 0}_{H,k} = \Rmsk g$ and $u^{\mathrm{ms}, 1}_{H,k} = \Rmsk g + \tau \Rmsk h + \tfrac{\tau^2}{2}D^2_tu(0)$ with $D^2_tu(0) = f(0)-M^{-1}Kg$. This can be shown by inserting these expressions into \eqref{eq:initial_data} and writing $u(t_1)$ as a Taylor expansion in terms of $u(0)$.
\end{remark}

\section{Numerical examples}\label{s:numerical_examples}
\begin{figure}
\centering
\includegraphics[width = 0.8\textwidth]{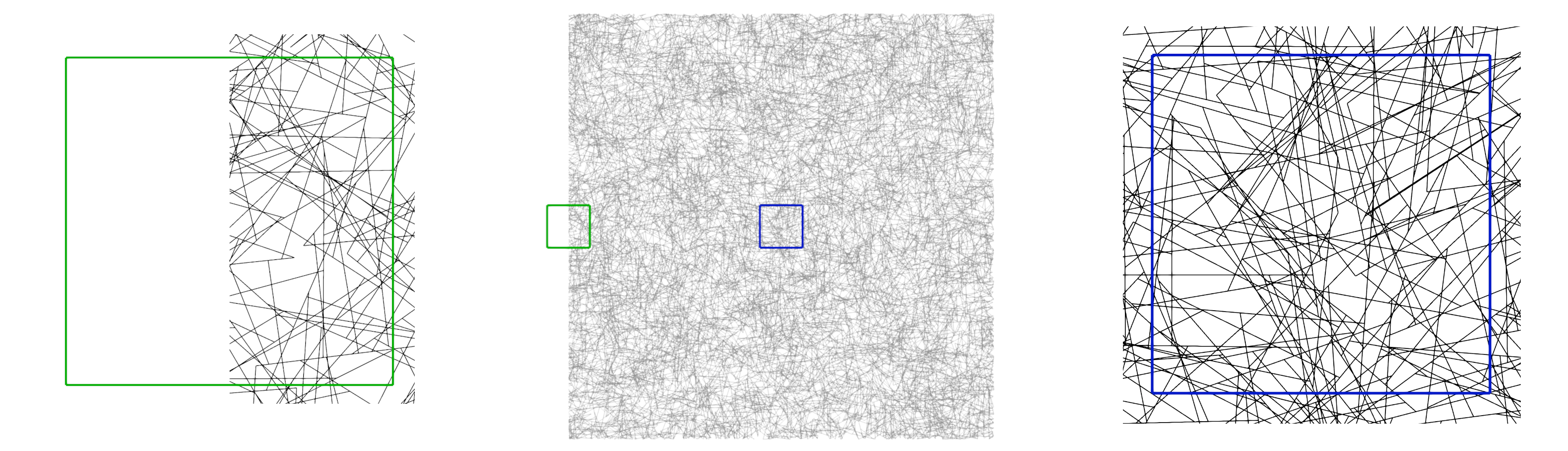}
\caption{The network used in the numerical examples. The entire network is shown in the center figure, with smaller areas (illustrated by the colored frames) presented in detail on the sides.}
\label{fig:numnetwork}
\end{figure}
The LOD method \eqref{eq:localized_method} and the error bound in Corollary~\ref{cor:comp_conv_rate} are evaluated numerically for three different problems on the form \eqref{eq:refmethod}. In all three numerical examples, the same two-dimensional network is used. The first two problems evaluate the method for the operator presented in Example~\ref{example:weighted_laplacian}. Both a homogeneous and an inhomogeneous wave equation are considered. The third example is related to elastic wave propagation. 

The network used in the numerical examples is created by randomly placing
straight line segments in the unit square. The network generation is performed in four steps. First, line segments of length $r=0.07$ are randomly placed with their midpoints in the extended domain $[-0.5r,1+0.5r]^2$, with a random rotation. The placement of the line segments and their orientation are chosen from a uniform distribution, and anything placed outside of the unit square is removed. These segments are added until the total length of the line segments placed is  \com{700}. After this step, the intersections of the line segments are found. These segments are then connected by placing nodes in all intersections, with any disconnected parts of the network removed. Then, any nodes closer than $ \com{10^{-3}}r$ are merged to get a lower bound on the lengths of the edges. The final step removes all edges that do not have a fixed endpoint and are connected at only one point to the bulk. This process results in networks with around \com{150 000} nodes. An illustration of the network is presented in Figure~\ref{fig:numnetwork}. In \cite{EdeGHKM22, GorHM22}, it is numerically shown that  \com{this type of network} satisfies Assumption~\ref{ass:network_properties}.

For all numerical examples, the LOD method used to approximate the problems has localization parameter $k=\log_2(1/H)$, and the performance is evaluated for grid sizes $H = 2^{-i}$, $i=2,3,4,5$. \com{In the numerical examples we use a slightly modified interpolation operator. When computing the nodal variables we evaluate $(M_{T_j}\psi_j,v)$ for all elements $T$ in the support of the basis function $\varphi_j$, rather than just picking one, and then take the average. Since the interpolation bound in Lemma~\ref{lem:interpolation_bound} holds for any choice of element $T_j$ in the construction of the nodal variables, it also holds for the average. This choice of interpolant gives better results when $H$ is fairly large, to a small additional cost.} 

\subsection{Scalar wave equation}

In the first two examples, the operator $K$ defined in Example~\ref{example:weighted_laplacian} is used with each edge coefficient, $\gamma_{xy}$, chosen at random in the interval $[0.1,0.9]$.

\subsubsection{The homogeneous wave equation}

In order to get a problem with a known exact solution, we pick the initial value of $u$ to be the solution to the corresponding generalized eigenvalue problem, $Kw=\lambda w$. We pick the sixth eigenvalue $\lambda_6\approx  \com{16}$ with corresponding eigenvector $w_6$. Through this construction, we also get well-prepared initial data. We consider

\begin{equation}\label{eq:eignum}
\begin{cases}
MD_t^2u + Ku = 0, \ t \in [0,\pi/\sqrt{\lambda_6}]\\
u = 0, \ x \in \Gamma = \{ x\in\mathcal{N} : x_1 \in \{0,1\} \},\\
u(0)  = w_6, \ D_tu(0) = 0.
\end{cases}
\end{equation}
The solution to this problem is $u(t) = \cos(\sqrt{\lambda_6}t)w_6$, which has half a period in the given time interval. An illustration of the exact solution in this time interval can be found in Figure~\ref{fig:eigsol}. 

\begin{figure}
\centering
\includegraphics  {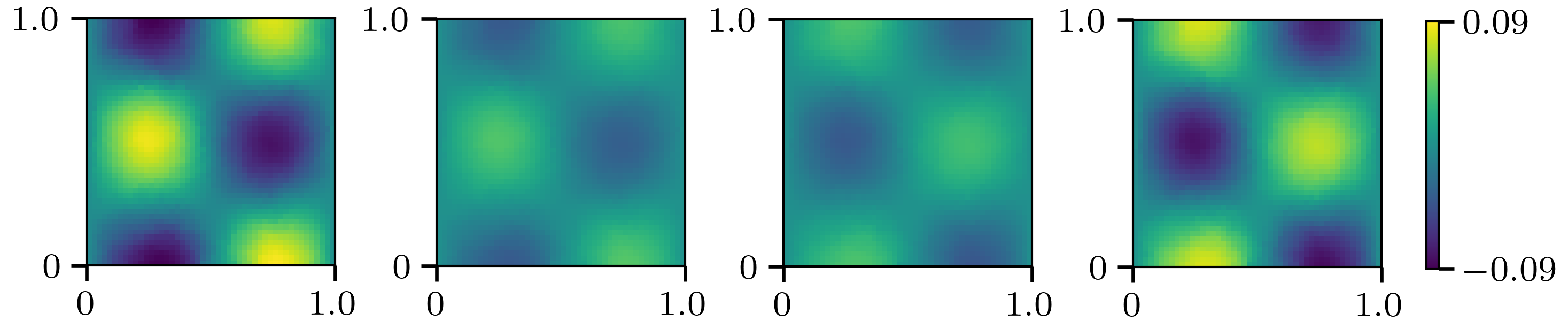}
\caption{The exact solution to \eqref{eq:eignum} at $t=\alpha \pi/\sqrt{\lambda_6}$, $\alpha = 0, \tfrac{1}{3}, \tfrac{2}{3}, 1$.}
\label{fig:eigsol}
\end{figure}

The solution to this problem is approximated using the LOD method \eqref{eq:localized_method} with a fine time discretization of $\tau = 10^{-3}$. The two initial conditions are chosen with the formula presented in Remark~\ref{remark:good_initial}, i.e.,
$$
u^{\mathrm{ms}, 0}_{H,k} = \Rmsk u(0) = \Rmsk w_6, \ \  
u^{\mathrm{ms}, 1}_{H,k} = \Rmsk \left(w_6-\lambda_6\frac{\tau^2}{2}w_6\right), 
$$
using that $M^{-1}Ku(0) = M^{-1}Kw_6 = \lambda_6M^{-1}Mw_6 = \lambda_6w_6$. The convergence results for these simulations are presented in Figure \ref{fig:eignumconv}. We see optimal order convergence in both $M$-norm and $K$-norm. Note that we present absolute values of the errors. In order to understand the relative size of the error compared to the solution, we note that the (preserved) energy of the solution fulfills
$$
|D_t u(t)|^2_M+|u(t)|^2_K=
|D_t u(0)|^2_M+|u(0)|^2_K=
|w_6|^2_K=\lambda_6 |w_6|^2_M=
\lambda_6\approx  \com{16}.
$$

\begin{figure}
\centering
\includegraphics [scale = 0.8]{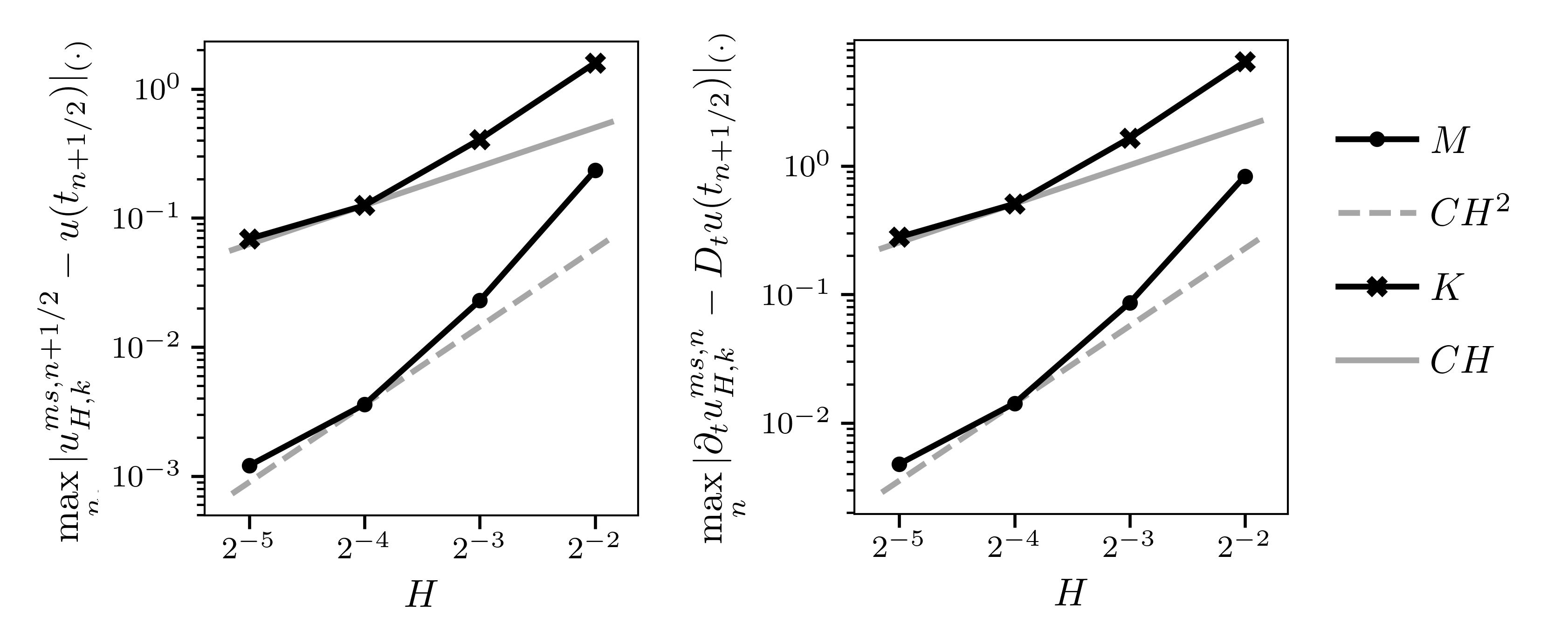}
\caption{The convergence of the LOD method. In the left figure, the error bound in Theorem~\ref{lem:error_k_norm} is observed with $H$ convergence in the $K$-norm, and in the right figure, the error bound in Theorem~\ref{lem:error_m_norm} is surpassed with $H^2$ convergence in the $M$-norm.}
\label{fig:eignumconv}
\end{figure}

\subsubsection{The inhomogeneous wave equation}
The second numerical example uses the same network and $K$ as in \eqref{eq:eignum}. Whereas the LOD approximations in the previous example were compared to an exact solution, the following problem is compared to a reference solution $\un$ computed on the full network discretization. We consider the problem

\begin{equation}\label{eq:gennum}
\begin{cases}
MD_t^2u + Ku = M\left(\sin(2\pi t)\cdot 1\right), \ t\in [0,2]\\
u = 0, \ x \in \partial\Omega,\\
u(0)  = 0, \ D_tu(0) = 0,
\end{cases}
\end{equation}
where $1 \in V$.

\begin{figure}
\centering
\includegraphics {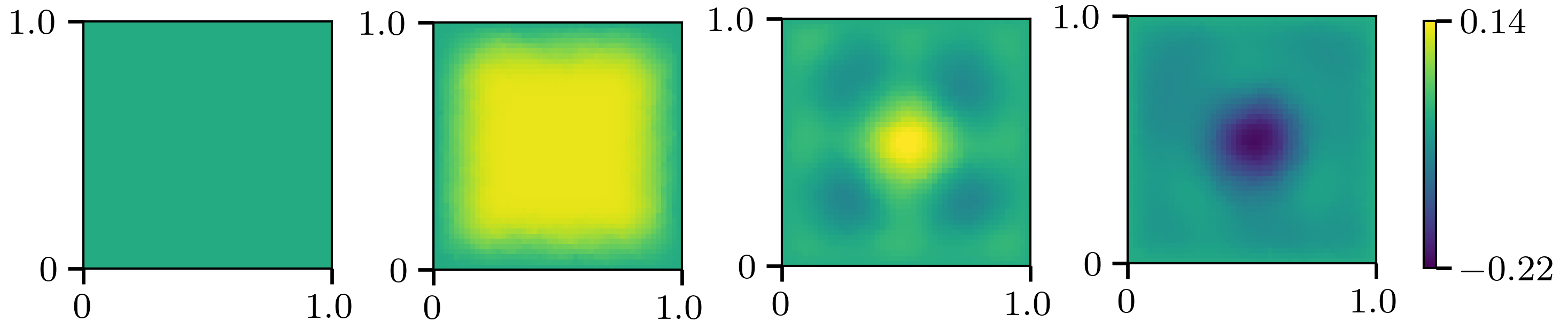}
\caption{The reference solution $u^n$ of \eqref{eq:gennum} at $t=0, \tfrac{2}{3}, \tfrac{4}{3}, 2$.}
\label{fig:hfig}
\end{figure}

The reference solution $u^n$ used for this problem is the discretization scheme presented in \eqref{eq:discretemethodref}. Again, the time discretization has high resolution with $\tau = 2 \cdot 10^{-3}$. This reference solution is presented in Figure~\ref{fig:hfig} at four different times. The problem is also approximated using the LOD method \eqref{eq:localized_method} with the same $\tau$. In the LOD method the initial conditions $u^{\mathrm{ms},0}_{H,k}=u^{\mathrm{ms},1}_{H,k}=0$ are used and motivated by Remark~\ref{remark:good_initial}. The convergence of the LOD method to the discrete solution $u^n$ can be found in Figure~\ref{fig:hconv}. Again we detect optimal order convergence in both $M$-norm and $K$-norm. 

\begin{figure}
\centering
\includegraphics [scale = 0.8] {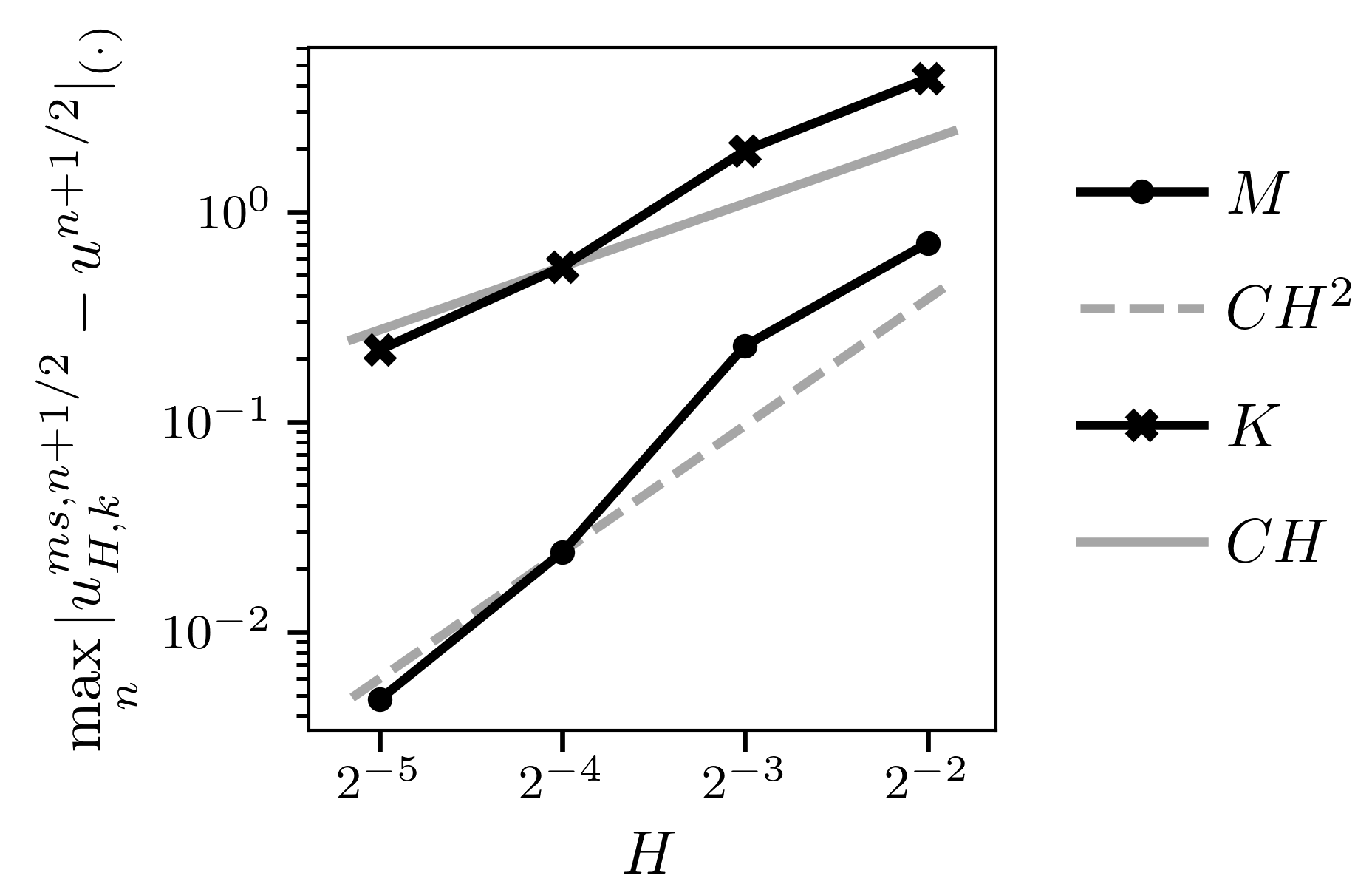}
\caption{The convergence of the LOD method to the reference solution $u^n$ of \eqref{eq:gennum}. In these results, it is evident that the LOD approximation converges linearly with respect to $H$ in the $K$-norm and quadratically in the $M$-norm.}
\label{fig:hconv}
\end{figure}

\subsection{Elastic wave equation}
An elasticity problem is evaluated in the last numerical example using the same network as in the previous two examples. In this example, the network is interpreted as a $\SI{1}{\meter}\times\SI{1}{\meter}$ mesh of steel cylindrical wires with radius $r_w =  \com{\SI{0.5}{\milli\meter}}$. The linear operator $\mathbf{K}:\mathbf{V}\rightarrow \mathbf{V}$ maps three-dimensional displacements of nodes to three-dimensional forces resulting from the displacement. Bold notation is used to emphasize vector-valued spaces, operators, and functions as presented in Remark \ref{remark:vec-not}, and for this specific case, the solutions space is $\textbf{V}=V^3$.

The operator $\mathbf{K}$ is composed of two additive parts, the first ($\textbf{KE}$) modeling the tensile forces resulting from strain and compression, and the second ($\textbf{KB}$) the forces resulting from bending the fibers. Here, we present the definitions of the operators. For a more detailed explanation of this model, we refer to \cite{KetMMFWE20}. 
 
\subsubsection{Elasticity model}

The tensile forces are modeled using a linearized version of Hooke's law as
\begin{equation} \label{eq:spring}
    (\textbf{KE}_x\textbf{v},\textbf{v}) =\frac{1}{2} \sum_{y \sim x} \gamma_{\text{KE}}  \frac{((\textbf{v}(x)-\textbf{v}(y))^T\partial_{xy})^2}{|x - y|},
\end{equation}
where $\partial_{xy} = (x-y)/|x-y|$ is the direction of the edge, and the coefficient $\gamma_{\text{KE}} = E\cdot A$ is Young's modulus of the wire times its cross-section area. In this example, Young's modulus is $\SI{210}{\giga\pascal}$.

For the second component, bending stiffness is modeled using a linearized Euler--Bernoulli model. This operator is composed of two parts, $\textbf{KB} = \textbf{KB}^{(1)} + \textbf{KB}^{(2)}$, where $\textbf{KB}^{(1)}$ models forces resulting from bending in the plane of the network, and $\textbf{KB}^{(2)}$ models forces resulting from bending out of the plane. These two operators are defined as
\begin{equation} \label{eq:fiber} 
  \begin{split}
    &(\textbf{KB}^{(k)}_x\textbf{v},\textbf{v}) = \\
    &\sum_{\substack{y\sim x \wedge z\sim x \\ y \not=z}} \gamma_{xyz} \frac{|x-y|+|x-z|} {2} \left(\frac{(\textbf{v}(y)-\textbf{v}(x))^T\eta_{xyz}^{y,(k)}}{|x-y|}+\frac{(\textbf{v}(z)-\textbf{v}(x))^T\eta_{xyz}^{z,(k)}}{|x-z|}\right)^2,
  \end{split}
\end{equation}
for $k = 1, \ 2$, where $\eta_{xyz}^{y,(2)} = \eta_{xyz}^{z,(2)}  = [0,0,1]^T$, $\eta_{xyz}^{y,1} = [0,0,1]^T \times \partial_{xy}, \ \eta_{xyz}^{y,1} = [0,0,1]^T \times \partial_{xz}$, and $\gamma_{xyz} = EI(|x-y|+|x-z|)^{-2}$ where $I = 0.25\pi r_w^4 = 0.25Ar_w^2$ is the second moment of area of the wire. 

The two types of coefficients $\gamma_{xy}$ and $\gamma_{xyz}$ satisfy the relation
 $$\gamma_{xyz} = EA\frac{r_w^2}{4(|x-y|+|x-z|)^{2}} = \gamma_{\text{KB}}\frac{r_w^2}{4(|x-y|+|x-z|)^{2}},$$
where the magnitude of the forces resulting from bending compared to tensile forces depends on the radius of the wire and the lengths of the edges, which is expected. 

With these operators defined, we can define the operator $\textbf{K}$ used in this numerical example as
\begin{equation} \label{eq:ellastic_operator}
  (\textbf{K} \textbf{v},\textbf{v}) =  \sum_{x\in\mathcal{N}} (\textbf{K}_x \textbf{v},\textbf{v}), \ \ (\textbf{K}_x \textbf{v},\textbf{v}) = (\textbf{KE}_x\textbf{v},\textbf{v}) +(\textbf{KB}^{(1)}_x\textbf{v},\textbf{v})+(\textbf{KB}^{(2)}_x\textbf{v},\textbf{v}).
\end{equation}

The coercivity assumption on $\textbf{K}$ for this operator is true depending on the geometry of the network, with the network required to be rigid. Moreover, at least $3$ nodes spanning a plane need to be in $\Gamma$. \com{We can compute $\alpha$ and $\beta$ numerically by considering the generalized eigenvalue problem $(\textbf{K}\textbf{v},\textbf{w})=\lambda(\textbf{L}\textbf{v},\textbf{w})$. By inverse power iteration and power iteration respectively, the smallest eigenvalue was found to be $\alpha = 11.6343$ and the largest $\beta = 642\ 657$. This spread can be compared to the coefficients  $\gamma_{\text{KE}} \approx 40 \ 000$ and $\gamma_{xyz}$ that are bounded by: 
$$13\leq \gamma_{xyz} = EA\frac{r_w^2}{4(|x-y|+|x-z|)^2} \leq 135\ 000,$$
using that $7\cdot10^{-5}\leq|x-y|\leq7\cdot10^{-3}$.}  An LOD method has been numerically validated for this model for elliptic problems in~\cite{EdeGHKM22}.

\subsubsection{The inhomogeneous elastic wave equation}
The following numerical example interprets the network as a mesh of steel wires fixed at one side with varying applied force. We consider the equation
\begin{equation}\label{eq:ellast_prob}
\begin{cases}
\textbf{M} D_t^2\textbf{u} + \textbf{K}\textbf{u} = \textbf{M}\left(\com{10^5x_1^2}\sin(\com{0.4}\pi t)\cdot \mathbf{1}_z\right), \ t\in [0,\com{10}]\\
\textbf{u} = 0, \ x \in \Gamma = \{ x\in\mathcal{N} :x_1 = 0\},\\
\textbf{u}(0)  = 0, \ D_t\textbf{u}(0) = 0,
\end{cases}
\end{equation}
with $\textbf{K}$ defined as in \eqref{eq:ellastic_operator}, $\textbf{M}$ represents an application of the mass matrix in each coordinate direction, and $\mathbf{1}_z(x) = [0,0,1]^T$. Similarly to the second numerical example, we measure the error with respect to a reference solution $\textbf{u}^n$ computed on the full network with the same time discretization. \com{In this example,  we set $\tau = 0.01$}. The reference solution $\textbf{u}^n$ is presented at three times in Figure~\ref{fig:SN}. 

The reference solution $\textbf{u}^n$ of \eqref{eq:ellast_prob} is compared to the approximation calculated using the LOD method \eqref{eq:localized_method} with the same $\tau$. Both the $\textbf{u}^{\mathrm{ms},0}_{H,k}$ and $\textbf{u}^{\mathrm{ms},1}_{H,k}$ terms vanish. 
The convergence rate between reference solution and approximation with respect to mesh size $H$, is depicted in Figure~\ref{fig:SNConv}. We again detect optimal order convergence in both $\textbf{M}$-norm and $\textbf{K}$-norm.

We conclude that the LOD method can be extended to the network setting and be used for efficient solutions to acoustic and elastic wave propagation problems posed on spatial networks. 

\begin{figure}
\centering
\includegraphics{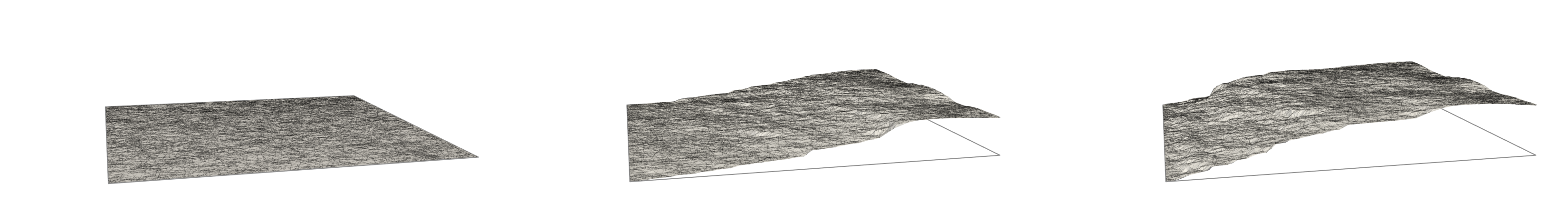}
\caption{Reference solution $\textbf{u}^n$ of \eqref{eq:ellast_prob} at $t=\com{0, 5, 10}$. The surface is the network model overlayed on a transparent triangularization for visualization purposes. The square outline is the boundary of the non-displaced network.}
\label{fig:SN}
\end{figure}

\begin{figure}
\centering
\includegraphics{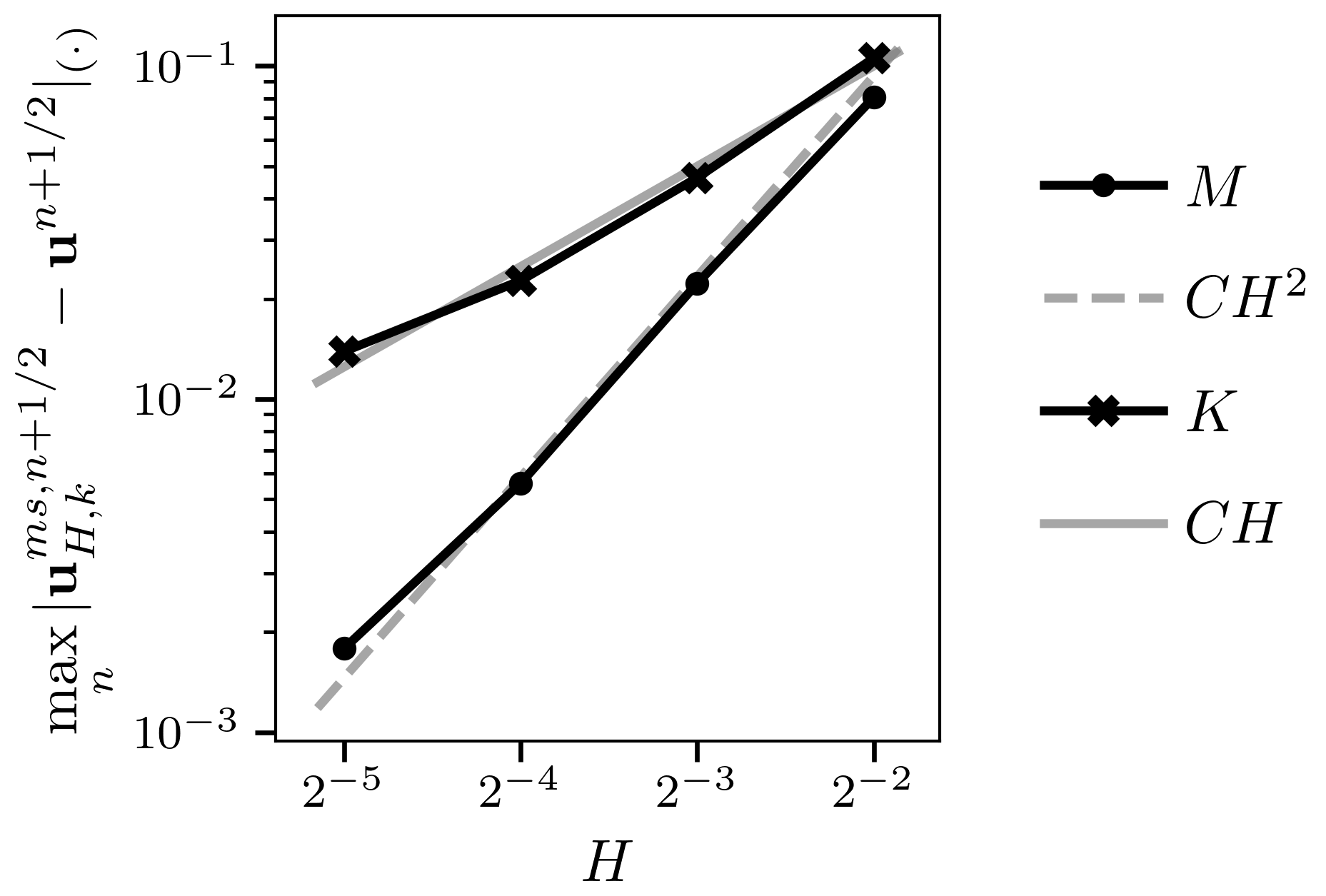}
\caption{The convergence of the LOD method to the reference solution $\textbf{u}^n$ of \eqref{eq:ellast_prob}. Similarly to the convergence results for the inhomogeneous problem \eqref{eq:gennum}, the LOD approximation converges linearly with respect to $H$ in the $\textbf{K}$-norm and quadratically in the $\textbf{M}$-norm.}
\label{fig:SNConv}
\end{figure}


\bibliographystyle{abbrv}
\bibliography{References}

\end{document}